\documentclass[12pt]{elsarticle}
\usepackage{mathptmx}
\usepackage[T1]{fontenc}
\usepackage{url}
\usepackage{bm}
\usepackage{amsmath}
\usepackage{amssymb}
\usepackage{graphicx}
\usepackage{esint}
\PassOptionsToPackage{normalem}{ulem}
\usepackage{ulem}

\makeatletter

\journal{Journal of Computational Physics}

\usepackage{amsthm}
\usepackage{algorithm}
\usepackage{algorithmicx}
\usepackage{algpseudocode}
\usepackage{tikz}
\usetikzlibrary{matrix}
\usepackage{tikz-3dplot}
\usepackage{pgfplots}
\usetikzlibrary{plotmarks}

\newtheorem{theorem}{Theorem}
\newtheorem{corollary}{Corollary}
\theoremstyle{definition}
\newtheorem{defn}{Definition}

\makeatother

\begin{document}
\begin{frontmatter}

\title{Mesh Adaptation on the Sphere using Optimal Transport and the Numerical Solution of a Monge-Amp\`ere type Equation}

\author[label1]{Hilary Weller}
\ead{h.weller@reading.ac.uk}
\author[label1]{Philip Browne}
\author[label2]{Chris Budd}
\author[label3]{Mike Cullen}
\address[label1]{Meteorology, University of Reading, UK}
\address[label2]{University of Bath, UK}
\address[label3]{Met Office, UK}

\begin{abstract}

Graphical abstract\\
\includegraphics[width=\linewidth]{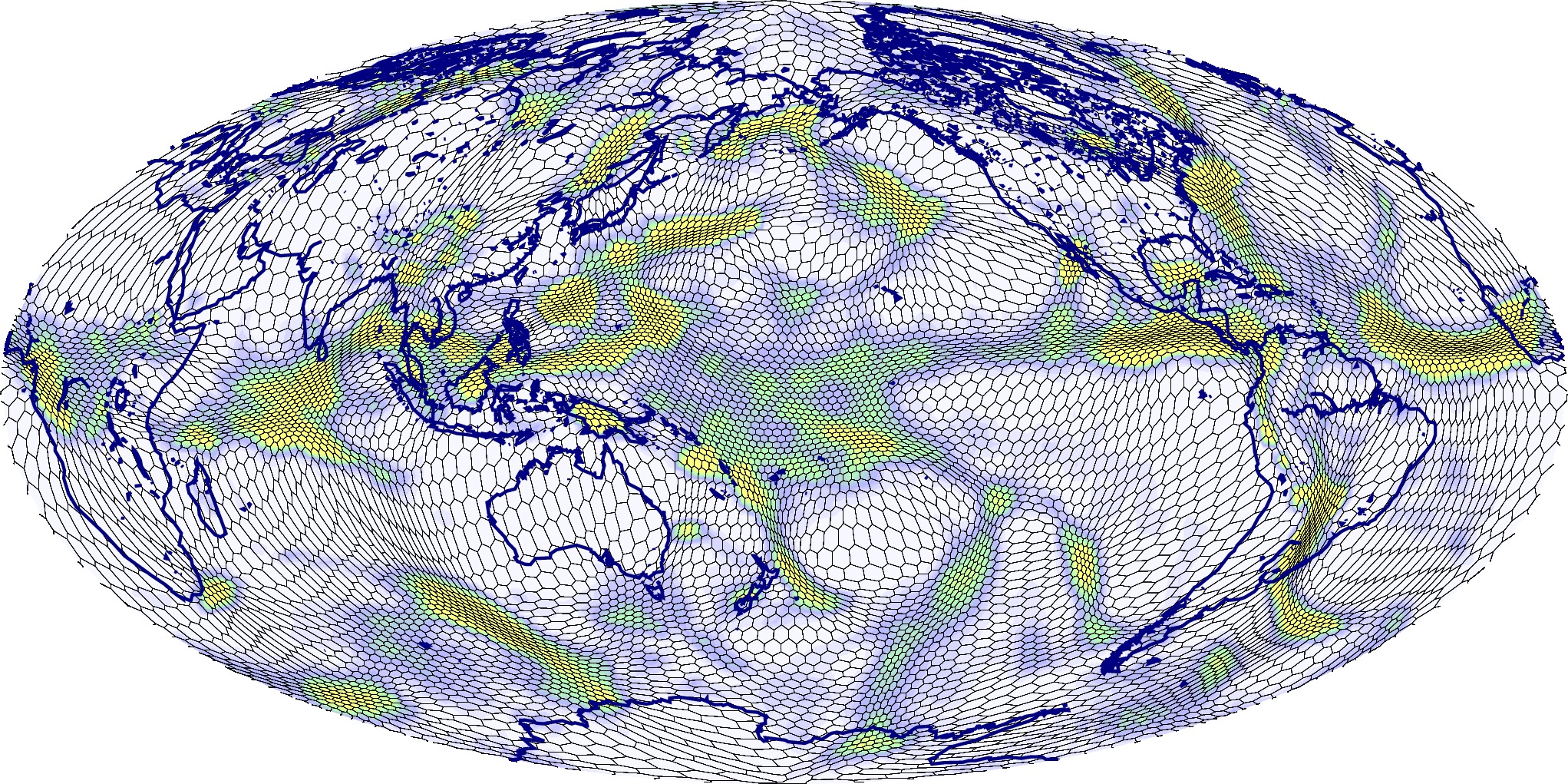}

An equation of Monge-Amp\`ere type has, for the first time, been solved numerically on the surface of the sphere in order to generate optimally transported (OT) meshes, equidistributed with respect to a monitor function. Optimal transport generates meshes that keep the same connectivity as the original mesh, making them suitable for r-adaptive simulations, in which the equations of motion can be solved in a moving frame of reference in order to avoid mapping the solution between old and new meshes and to avoid load balancing problems on parallel computers. 

The semi-implicit solution of the Monge-Amp\`ere type equation involves a new linearisation of the Hessian term, and exponential maps are used to map from old to new meshes on the sphere. The determinant of the Hessian is evaluated as the change in volume between old and new mesh cells, rather than using numerical approximations to the gradients. 

OT meshes are generated to compare with centroidal Voronoi tesselations on the sphere and are found to have advantages and disadvantages; OT equidistribution is more accurate, the number of iterations to convergence is independent of the mesh size, face skewness is reduced and the connectivity does not change. However anisotropy is higher and the OT meshes are non-orthogonal.

It is shown that optimal transport on the sphere leads to meshes that do not tangle. However, tangling can be introduced by numerical errors in calculating the gradient of the mesh potential. Methods for alleviating this problem are explored. 

Finally, OT meshes are generated using observed precipitation as a monitor function, in order to demonstrate the potential power of the technique.

\end{abstract}


\end{frontmatter}

\section{Introduction}

The need to represent scale interactions in weather and climate prediction
models has, for many decades, motivated research into the use of adaptive
meshes \citep{BO84,SK93,Wan01}. R-adaptivity - mesh \uline{r}edistribution
- involves deforming a mesh in order to vary local resolution and
was first considered for atmospheric modelling more than twenty years
ago by \citet{DD92}. It is an attractive form of adaptivity since
it does not involve altering the mesh connectivity, does not create
load balancing problems because points are never created or destroyed,
does not require mapping of solutions between meshes \citep{LTZ01},
does not lead to sudden changes in resolution and can be retro-fitted
into existing models. Variational methods exist which attempt to control
resolution in different directions for r-adaptive \marginpar{Rev2.2}meshes
\citep[eg][]{Hua01,KSD12}. Alternatively, the solution of the Monge-Amp\`ere
equation to generate an optimally transported (OT) mesh based on a
scalar valued monitor function is a useful form of r-adaptive mesh
generation because it generates a mesh equidistributed with respect
to a monitor function and does not lead to mesh tangling \citep{BHR09}.
We will see that the optimal transport problem on the sphere leads
to a slightly different equation of Monge-Amp\`ere type, which has
not before been solved numerically on the surface of a sphere, which
would be necessary for weather and climate prediction using r-adaptivity.

At first glance, r-adaptivity does not look ideal for adaptive meshing
of the global atmosphere; \citet{DD92} pointed out that the resulting
meshes can be quite distorted which leads to truncation errors and
it is not always possible to control the resolution in individual
directions, just the total cell size (area or volume); with r-adaptivity,
it is not possible, for example, to increase the total number of points
around the equator, just re-distribute them \citep{FT93}. However,
if the mesh redistribution starts from a mesh with enough points around
the equator, then these points can be redistributed according to transient
features of the flow.

Models of the global atmosphere are being developed with accurate
treatment of non-orthogonality and which allow arbitrary grid structures
\citep{GR04,MTN+11,CS12,MSC14,Wel14}. The time may therefore be right
to reconsider r-adaptive modelling of the global atmosphere.

A powerful form of adaptivity that, like r-adaptivity, retains the
same total number of points, is centroidal Voronoi tesselation using
a non-uniform density (or monitor) function to control the mesh spacing,
using Lloyd's algorithm \citep{RJG+11}. Lloyd's algorithm generates
smoothly varying, orthogonal, near centroidal isotropic meshes suitable
for finite-volume models and is being used by the Model for Prediction
Across Scales \citep[MPAS,][]{SKD+12}. Lloyd's algorithm alters the
mesh connectivity meaning that, if it is used in conjunction with
dynamic mesh adaptivity, mapping between old and new solutions is
needed and there is an additional layer of complexity involved with
changing the data structures and moving information between parallel
processors. Also, Lloyd's algorithm is extremely expensive, using
an explicit solution to find an equidistributed mesh - an elliptic
problem. The cost per iteration is proportional to the number of points,
$N$, \citep{JGR+13} and, in one dimension, the number of iterations
is proportional to $N$ \citep{DFG99}. Therefore, overall, the cost
is proportional to $N^{2}$. Conversely, generating optimally transported
meshes \marginpar{Rev3.1}using a semi-implicit technique, has convergence
independent of the mesh size and the overall cost is proportional
to $N\log N$ \citep{BBPC14}. We therefore propose r-adaptivity which
uses cheaper mesh generation and fixed data structures associated
with the mesh.

In section \ref{sec:OT} we describe mesh generation by optimal transport
in Euclidean space leading to a Monge-Amp\`ere equation. We then
show how these concepts can extend to mesh generation on the sphere,
leading to an equation of Monge-Amp\`ere type. Existing numerical
solution techniques in Euclidean geometry are reviewed in section
\ref{sec:review}. In section \ref{sec:method} we describe the new
numerical methods for solving the Monge-Amp\`ere type equations,
both on a Euclidean plane and on the sphere. In order to address issues
of mesh distortion, a range of diagnostics of mesh quality are presented.
These diagnostics, along with the diagnostics of solution convergence,
are described in section \ref{sec:diags} and the diagnostics of the
meshes generated are presented in section \ref{sec:results}. The
meshes generated, both on the plane and on the sphere, are shown and
described in section \ref{sec:results} and the meshes on the sphere
are compared with centroidal Voronoi meshes generated using Lloyd's
algorithm \citep{RJG+11} with the same monitor function. In order
to demonstrate the performance of the mesh generation using real data
as a monitor function, meshes are generated using a monitor function
derived from reanalysis precipitation in section \ref{sec:results}.
Final conclusions and recommendations for future work are drawn in
section \ref{sec:concs}.

\section{Mesh Generation by Optimal Transport (OT) \label{sec:OT}}

\subsection{Optimally Transported Meshes in Euclidean Space}

A mesh is equidistributed with respect to a monitor function when
the product of the cell volumes and the monitor function in the cell
is constant across all mesh cells. The equidistribution principle
alone does not lead to a well-defined problem for mesh generation.
Indeed this problem is ill-posed in more than 1 dimension and so requires
the imposition of an extra constraint. \citet{BW06} introduced optimal
transport for mesh generation to find a map from the original mesh
(or computational space, $\Omega_{c}$) to the new mesh (or physical
space, $\Omega_{p}$). This technique was further developed by \citet{BHR09}
and extended to 3 spatial dimensions by \citet{BBPC14}. The \textit{optimal
transport }constraint says that the new mesh should be as close as
possible to the original mesh - we seek to minimise the distance between
the two meshes in a certain measure which we shall discuss. We write
this minimization problem:
\begin{equation}
\min_{\mathbf{x}\in\Omega_{p}}d(\bm{\xi},\mathbf{x})^{2}
\end{equation}
where $d$ is the distance metric between the two meshes and $\bm{\xi}\in\Omega_{c}$
maps to $\mathbf{x}\in\Omega_{p}$. In Cartesian space $[0,1]^{n}$
this metric can simply be the sum of the Euclidean distance between
all of the corresponding points defining the meshes. Brenier's theorem
\citep{brenier1991} then tells us that the unique, optimal transport
map from $\mathbf{x}$ to $\bm{\xi}$ is the gradient of a convex
scalar potential, $\phi$, so that the new mesh locations are given
by: 
\begin{equation}
{\bf x}=\bm{\xi}+\nabla\phi.\label{eq:meshMap}
\end{equation}
The change in cell volume under the coordinate transform is given
by the determinant of the Jacobian of the map, $|J(\xi)|=|\mathbf{\nabla\mathbf{x}}(\bm{\xi})|$,
the gradient of $\mathbf{x}$ with respect to $\bm{\xi}$. Therefore,
for equidistribution with respect to a monitor function, $m$, the
new mesh locations should satisfy
\begin{equation}
|\nabla\mathbf{x}|\ m\left(\mathbf{x}\right)=c\label{eq:equiDist}
\end{equation}
where $c$ is a constant, uniform over space, which will be determined
once the numerical method is defined. Taking the determinant of the
gradient of eqn. (\ref{eq:meshMap}), we can see that $|\nabla\mathbf{x}|=|I+\nabla\nabla\phi|=|I+H\left(\phi\right)|$
where $I$ is the identity tensor and $H$ is the Hessian. Consequently,
for the mesh to be optimally transported and equidistributed, the
mesh potential, $\phi$, must satisfy a Monge-Amp\`ere equation:
\begin{equation}
|I+H\left(\phi\right)|m({\bf x})=c.\label{eq:MA}
\end{equation}
The presence of the identity tensor in this Monge-Amp\`ere equation
will be exploited in the linearisation to create a novel numerical
algorithm. 

Mesh tangling is caused by a local loss of invertibility of the Jacobian
of the map from the original to the tranported mesh. \marginpar{Rev3.2}Given
that the solution of the Monge-Amp\`ere equation, $\phi$, is convex,
the determinant of the Hessian of $\phi$ is positive and hence the
Jacobian determinant of the map is positive and thus is invertible
and the mesh will not tangle \citep{BHR09}.

\subsection{Optimally Transported Meshes on the Sphere}

A naive approach to r-adaptivity on the sphere, $\mathbb{S}^{2}$,
would be to map the surface onto the plane, use an established method
to solve a mesh redistribution problem on the plane, then map back
to the sphere. As shown in Figure \ref{commutative_diagram}, the
desired map $T$ could be written as a composition of mappings as
$T=g^{-1}\circ t\circ g$.

\begin{figure}
\begin{centering}
\begin{tikzpicture}
  \matrix (m) [matrix of math nodes,row sep=3em,column sep=4em,minimum width=2em]
  {
     \mathbb{S}^2 & \mathbb{S}^2 \\
     \mathbb{R}^2 & \mathbb{R}^2 \\};
  \path[-stealth]
    (m-1-1) edge node [left] {$g$} (m-2-1)
    (m-1-1) edge [double] node [below] {$T$} (m-1-2)
    (m-2-2) edge node [right] {$g^{-1}$} (m-1-2)
    (m-2-1) edge [double] node [below] {$t$} (m-2-2);
\end{tikzpicture}
\par\end{centering}

\caption{%
Commutative diagram showing an naive approach to meshing on the sphere
by converting the problem to the plane%
}

\label{commutative_diagram}
\end{figure}

A map $g:\mathbb{S}^{2}\to\mathbb{R}^{2}$ must be chosen and an optimal
transport map $t$ found. The boundary conditions for the problem
of finding $t$ must be specified, and those boundary conditions would
necessarily depend on $g$. For example in the case where the mapping
$g$ is simply the lat-lon decomposition of $\mathbb{S}^{2}$, the
boundary conditions for the mesh redistribution problem on the plane
will then be periodic in the zonal direction. In the the meridional
direction, Neumann boundary conditions would not be appropriate as
the poles will not be free to move and they will be mapped back to
their original location under $g^{-1}$.

The Hairy Ball Theorem tells us that there must be at least one fixed
point of the map $T:\mathbb{S}^{2}\to\mathbb{S}^{2}$. The decomposition
$T=g^{-1}\circ t\circ g$ would then be possible if $g$ maps the
fixed points of $T$ to a Neumann boundary of $\mathbb{R}^{2}$. However
the location of the fixed points of $T$ are not known \textit{a priori},
and hence choosing $g$ appropriately would form a significant problem
by itself. Hence we will seek a direct optimal transport map, $T:\mathbb{S}^{2}\to\mathbb{S}^{2}$
which will be described in this section.

On the surface of the sphere, we would still like to define an optimally
transported mesh satisfying equidistribution:
\begin{equation}
r\left(\phi\right)m\left(\bm{x}\right)=c\label{eq:MAsph}
\end{equation}
\marginpar{Rev3.3}where $r=V_{\xi}/V_{x}$ is the ratio of the volumes
of the original mesh cells, $V{}_{\xi}$, with vertices at positions
$\bm{\xi}$, and the volumes of the new mesh cells, $V_{x}$, with
vertices at positions $\bm{x}$. We need to ascertain if unique solutions
of (\ref{eq:MAsph}) exist which minimise the distance between the
original and resulting meshes. On the sphere $\mathbb{S}^{2}$, the
appropriate distance metric is the Riemannian distance on the surface
of the sphere between all of the corresponding points defining the
meshes. We cannot use Brenier's theorem on the sphere. Instead, we
appeal to the generalised version of Brenier's theorem given by \citet{McC01},
a detailed discussion of which is given in \citet{villani2003}.

\begin{defn}{[}$c$-convex function{]} The $c$-transform $\phi^{c}$
of a function $\phi:\mathbb{S}^{2}\to\mathbf{\mathbb{R}}$ is defined
as 
\begin{equation}
\phi^{c}(y)=\sup_{\bm{\xi}\in\mathbf{\mathbb{S}}^{2}}\{-c(\bm{\xi},\mathbf{x})-\phi(\bm{\xi})\}.
\end{equation}
The function $\phi$ is said to be \emph{$c$-convex}, or cost-convex,
if $(\phi^{c})^{c}=\phi.$ \end{defn}

\begin{theorem}\label{thm:mccannthm} \marginpar{Ed4}(A combination
of theorems 8 and 9 of\citep{McC01}) Let $M$ be a connected, complete
smooth Riemannian manifold, equipped with its standard volume measure
$\mathrm{d}x$. Let $\mu,\nu$ be two probability measures on $M$
with compact support, and let the objective function $c(\bm{\xi},\mathbf{x})$
be equal to $d(\bm{\xi},\mathbf{x})^{2}$, where $d$ is the geodesic
distance on $M$. Further, assume that $\mu$ is absolutely continuous
with respect to the volume measure on $M$. Then, the Monge-Kantorovich
mass transportation problem between $\mu$ and $\nu$ admits a unique
optimal transported map $T$ where $T$ pushes forward the measure
$\mu$ onto $\nu$. Then, (using classical optimal transport notation):
\begin{equation}
T_{\#}\mu=\nu\label{eq:t1}
\end{equation}
such that 
\begin{equation}
\mathbf{x}=T(\bm{\xi})=\exp_{\bm{\xi}}[\nabla\phi(\bm{\xi})]\label{eq:t2}
\end{equation}
for some $d^{2}/2$-convex potential $\phi$. \end{theorem}

\begin{corollary} There exists a unique, optimally transported mesh
on the sphere that satisfies the equidistribution principle. Moreover,
that mesh is defined by a $c$-convex scalar potential function that
\marginpar{Rev1.2} satisfies the Monge-Amp\`ere type equation
\begin{equation}
m\left(\exp_{\xi}[\nabla\phi(\bm{\xi})]\right)|J(\bm{\xi})|=c.
\end{equation}
 \end{corollary}

\begin{proof} Clearly $M=\mathbb{S}^{2}$ satisfies the conditions
on $M$ in Theorem \ref{thm:mccannthm}. The first probability measure
of interest, $\mu$, we define to be the scaled Lebesgue measure such
that:
\begin{equation}
\mathrm{d}\mu=\frac{\mathrm{d}x}{\int_{\mathbb{S}^{2}}\mathrm{d}x}.
\end{equation}
The target probability measure, $\mu$, is the Lebesgue measure appropriately
scaled by the monitor function to be equidistribed, such that:
\begin{equation}
\mathrm{d}\nu=\frac{m(\mathbf{x})\mathrm{d}x}{\int_{\mathbb{S}^{2}}m(\mathbf{x})\ \mathrm{d}x}.
\end{equation}
As $M=\mathbb{S}^{2}$ these are trivially compactly supported. $\mu$
is absolutely continuous. Hence by Theorem \ref{thm:mccannthm} we
have that there exists a unique solution, $T$, to the mass transportation
problem between $\mu$ and $\nu$. From (\ref{eq:t2}) we can see
that any point in the new mesh, $\mathbf{x}$, is defined by the action
of the exponential map on the scalar potential, $\phi$.

To see that this map, $T$, will give a mesh that satisfies the equidistribution
principle, consider a cell $A_{\xi}$ in the original computational
mesh, $\Omega_{C}$ with volume $V_{\xi}$. The mapping of the cell
under $T$ gives the new cell, $A_{x}$ in the physical mesh $\Omega_{p}$.
As $T$ is a (optimal) transport map, then the integral over a set
with respect to the measure $\mu$ equals the integral over the image
of that set with respect to $\nu$. Hence:
\begin{equation}
\int_{A_{\xi}}\mathrm{d}\mu=\int_{A_{x}}\mathrm{d}\nu\implies\frac{V_{\xi}}{\int_{\mathbb{S}^{2}}\mathrm{d}x}=\frac{\int_{A_{x}}m(\mathbf{x})\ \mathrm{d}x}{\int_{\mathbb{S}^{2}}m(\mathbf{x})\ \mathrm{d}x}.
\end{equation}
The ratio of the integral of the monitor function over the new cell
with the total integral of the monitor function is equal to the proportion
of the volume that the original cell occupied in the original mesh.
This is precisely what it means for the monitor function to be equidistributed
on a discretised mesh. 

Using a change of variables, we have:
\begin{equation}
\frac{V_{\xi}}{\int_{\mathbb{S}^{2}}\mathrm{d}x}=\frac{\int_{A_{x}}m(\mathbf{x})\ \mathrm{d}x}{\int_{\mathbb{S}^{2}}m(\mathbf{x})\ \mathrm{d}x}=\frac{\int_{A_{\xi}}m(\exp_{\xi}[\nabla\phi(\bm{\xi})])|J(\bm{\xi})|\ \mathrm{d}\xi}{\int_{\mathbb{S}^{2}}m(\mathbf{x})\ \mathrm{d}x}\label{eq:cov}
\end{equation}
where $|J(\bm{\xi})|$ is the determinant of the Jacobian of the map
$T(\bm{\xi})=\exp_{\xi}[\nabla\phi(\bm{\xi})]$.

As (\ref{eq:cov}) must hold for arbitrary sets $A_{\xi}\in\Omega_{C}$,
the following equation of Monge-Amp\`ere type on the sphere results:
\begin{equation}
m\left(\exp_{\xi}[\nabla\phi(\bm{\xi})]\right)|J(\bm{\xi})|=\frac{\int_{\mathbb{S}^{2}}m(\mathbf{x})\ \mathrm{d}x}{\int_{\mathbb{S}^{2}}\mathrm{d}x}=c.
\end{equation}
\end{proof}

\begin{corollary} The optimally transported mesh on the sphere satisfying
the equidistribution principle does not exhibit tangling. \end{corollary}
\begin{proof} The choice of cost function $c$ to be the squared
geodesic distance is crucial to the proof of uniqueness in Theorem
\ref{thm:mccannthm}. Indeed simply taking $c$ to be the square of
the Euclidean distance is not sufficient \citep{ahmad2004}. The squared
geodesic distance is necessary to ensure that the classical \emph{twist}
condition holds, i.e. $T$ given in \eqref{eq:t1} is injective and
hence is a map.

The injectivity of this map ensures that \eqref{eq:t2} is locally
invertible, i.e. for each point in the new mesh, $\mathbf{x}$, there
is a unique point in the original mesh, $\bm{\xi}$, which maps to
it - i.e. mesh tangling is not present. \end{proof}

\section{A Review of Numerical Methods for solving the Monge-Amp\`ere Equation\label{sec:review}}

The fully non-linear, second-order, elliptic Monge-Amp\`ere equation
is:
\begin{equation}
|H\left(\phi\left(\bm{\xi}\right)\right)|=f\left(\bm{\xi},\phi\right)\label{eq:MA1}
\end{equation}
for independent variable $\bm{\xi}\in\Omega$ and $\Omega\subset\mathbb{R}^{d}$
where $\phi$ is the (scalar) dependent variable, $f$ is a known
scalar function of $\bm{\xi}$ and $\phi$, $H=\nabla\nabla$ is the
Hessian (the tensorial gradient of the gradient) and $|H|$ is the
determinant of the Hessian. \citet{FO11} give an excellent review
of some numerical methods for solving this equation and this review
draws from and adds to their review.

There are two challenging parts to solving the Monge-Amp\`ere equation.
Firstly we need spatial discretisation methods both for the Hessian,
$H$, and for the source term, $f$ (although $f$ is a known function,
it can be a function of $\phi$ or of $\nabla\phi$, so numerical
approximations are necessary). The spatial discretisation leads to
a set of non-linear algebraic equations. Secondly, the algebraic equations
require a numerical algorithm to find solutions. \marginpar{Rev1.3}The
discretisation should ensure that the solutions is convex based on
a discrete definition of convexity. We will start by considering the
spatial discretisation of the Hessian, $H$. 

\citet{BHR09} used finite differences on a structured, Cartesian
grid to discretise the Hessian, a technique that was extended to three
dimensions by \citet{BBPC14}. \marginpar{Rev1.3}Convexity was ensured
by filtering the monitor function and smoothing the non-converged
solution. \citet{Obe08} describe a finite difference method that
uses a wide stencil to calculate the Hessian on a structured Cartesian
grid. This was extended to three dimensions by \citet{FO11}. The
wide stencil was needed to ensure monotonicity of the iterative solution
to the convex, numerical solution. \citet{FOS1x} study the slightly
different, 2-Hessian equation and describe how they rotate the coordinate
system so that the Hessian is diagonal and hence the solution is convex
and the discretisation is monotone. \citet{FN09} approximate the
Monge-Amp\`ere equation by a fourth-order quasi-linear equation in
order to use mixed finite elements for the spatial discretisation.
\citet{DG06,DG06b} also use mixed finite-elements on triangulations
of the unit plane. All of these techniques have been used on either
2D or 3D Euclidean geometry.

\marginpar{Rev1.4}When solving the Monge-Amp\`ere equation for mesh
adaptation, the RHS of eqn (\ref{eq:MA1}) depends on $\nabla\phi$.
\citet{FroeseThesis2012} pointed out that standard centred differences
are not monotone for discretising this term and so used wide stencil
finite differences. \citet{SAK1x} experimented with second and fourth
order centred finite differences and a spectral method for discretisting
$\nabla\phi$.

Once the Monge-Amp\`ere equation is discretised in space, it is necessary
to solve the resulting non-linear algebraic equations, the part of
the method that we describe as the ``algorithm''. \citet{BW06,BHR09}
introduced a parabolic version of the Monge-Amp\`ere equation which
is solved by time-stepping, including an implicit relaxation term
to smooth the transient solution and to speed up convergence:
\begin{equation}
\left(I-\gamma\nabla^{2}\right)\frac{\partial\phi}{\partial t}=\left(m\left(\nabla\phi\right)|I+H\left(\phi\right)|\right)^{\frac{1}{d}}\label{eq:PMA}
\end{equation}
where $\gamma$ is a scalar parameter defining the amount of smoothing
applied, $\nabla^{2}$ is the Laplacian operator and $d$ is the number
of spatial dimensions. The time-stepping effectively creates fixed-point
iterations but it may be possible to create more convergent iterations,
without smoothing towards a uniform mesh. \citet{BFO10} also used
fixed-point iterations by linearising the two-dimensional Hessian
term with a Laplacian:
\begin{equation}
|H\left(\phi\right)|=\frac{1}{2}\left(\nabla^{2}\phi\right)^{2}-\frac{\phi_{xx}^{2}+\phi_{yy}^{2}}{2}-\phi_{xy}^{2}.
\end{equation}
After some manipulation, this results in a Poisson equation which
can be solved implicitly with the non-linear terms on the right hand
side. \citet{FO11} describe this as a semi-implicit method and use
it to find the starting point for a Newton method. A Newton method
is a common algorithm for solving the algebraic equations \citep{DG06,FO11,CSC14}.
However the cost and complexity of the Newton method may not be necessary
for mesh generation. In this paper we focus on fixed-point iterations,
although the new linearisation proposed may also be beneficial for
calculating the first guess for a Newton method. 

Equations of Monge-Amp\`ere type have not before been solved numerically
on the sphere. The description of the optimally transported mesh problem
using the Monge-Amp\`ere equation relies on properties of Euclidean
geometry \citep{BHR09}. The numerical solution technique for the
optimal transport problem on the surface of a sphere will be described
in section \ref{sec:method}.

\section{Numerical Method for Calculating OT Meshes\label{sec:method}}

There are two aspects to solving equations of Monge-Amp\`ere type
in order to calculate optimally transported (OT) meshes. The spatial
discretisation describes how to calculate the gradient and the Hessian
of the mesh potential, $\phi$, from discrete values (in this instance
in finite volume cells). This will convert the PDE into a set of non-linear
algebraic equations. The algorithm describes how to linearise and
solve the large set of non-linear algebraic equations.

\subsection{The Numerical Algorithm}

\subsubsection{In Euclidian Space\label{sub:linearPlane}}

A fixed-point iteration sequence to solve eqn. (\ref{eq:MA}) can
be found by observing that the linear terms of $|I+H\left(\phi\right)|$
are in fact $1+\nabla^{2}\phi$ where $\nabla^{2}$ is the Laplacian
operator \marginpar{Rev3.4}(linearising about $\phi=0$). Eqn. (\ref{eq:MA})
can then be written as fixed-point iterations:
\begin{equation}
1+\nabla^{2}\phi^{n+1}=1+\nabla^{2}\phi^{n}-|I+H\left(\phi^{n}\right)|+\frac{c^{n}}{m({\bf x}^{n})}\label{eq:MA_SI}
\end{equation}
where $n$ is the iteration number and where:
\begin{equation}
{\bf x}^{n}=\bm{\xi}+\nabla\phi^{n}.
\end{equation}
This is simpler than the fixed-point iterations used by \citet{FN09,BFO10}
because of the presence of the identity tensor in our Monge-Amp\`ere
equation which simplifies the linearisation. These fixed-point iterations
are similar to the solution of the parabolic Monge-Amp\`ere equation
by \citet{BBPC14} but could have advantages because the Laplacian
term should initially accelerate convergence whereas the Laplacian
smoothing used by \citet{BBPC14} was only used to smooth intermediate
iterations. 

Given suitable spatial discretisations, eqn. (\ref{eq:MA_SI}) can
be solved for $\phi^{n+1}$ given known values $\phi^{n}$. Assuming
periodic boundary conditions, for the Poisson equation (\ref{eq:MA_SI})
to have a solution, $c^{n}$ must take the value
\begin{equation}
c^{n}=\frac{\sum|I+H\left(\phi^{n}\right)|V_{\xi}}{\sum\frac{V_{\xi}}{m\left(\mathbf{x}^{n}\right)}}
\end{equation}
where $V_{\xi}$ are the volumes of the original, computational mesh
cells and the summations are over all cells of the computational mesh.

Without a monotone spatial discretisation, numerical solutions of
the Monge-Amp\`ere equation can become non-convex leading to artificial
oscillations in the numerical solution and non-convergence \citep{FOS1x}.
The spatial discretisation described here is not monotone and the
numerical solution can become non-convex. Therefore, in order to improve
stability of the fixed-point iteration sequence, the Laplacian terms
of eqn. (\ref{eq:MA_SI}) can be multiplied by a factor, $1+\alpha$,
where $\alpha\ge0$:
\begin{equation}
\left(1+\alpha\right)\nabla^{2}\phi^{n+1}=\left(1+\alpha\right)\nabla^{2}\phi^{n}-|I+H\left(\phi^{n}\right)|+\frac{c^{n}}{m({\bf x}^{n})}\label{eq:PMA_SI}
\end{equation}
which clearly has no affect on a converged solution but will alter
the convergence of the fixed-point iterations used to find $\phi$
and can help to keep the numerical solution smooth. This is a form
of under-relaxation and the value of $\alpha$ will be defined in
section \ref{sub:underRelax}. A solution of the Monge-Amp\`ere equation
only controls $|I+H\left(\phi\right)|$, not the individual eigenvalues
of $I+H\left(\phi\right)$. If one of the eigenvalues gets large and
the other small, the Laplacian preconditioning will\marginpar{Rev2.10a}
not lead to convergent iterations without the under-relaxation.

\subsubsection{On the Surface of the Sphere}

In order to solve the optimal transport problem on the sphere, we
solve eqn (\ref{eq:MAsph}) directly rather than eqn (\ref{eq:MA}).
However, in order to define fixed-point iterations, we need to find
a linearisation of eqn. (\ref{eq:MAsph}). \marginpar{Rev1.1}For
small maps, we assume that maps lie on a tangent to the sphere and
so eqn. (\ref{eq:MAsph}) can be approximated by eqn. (\ref{eq:MA}).
We then use the same linearisation as in section \ref{sub:linearPlane}
and the same fixed-point iteration sequence:
\begin{equation}
\left(1+\alpha\right)\nabla\phi^{n+1}=\left(1+\alpha\right)\nabla^{2}\phi^{n}-r\left(\phi^{n}\right)+\frac{c^{n}}{m({\bf x}^{n})}\label{eq:PMA_SIsph}
\end{equation}
where
\begin{equation}
\mathbf{x}^{n}=\exp_{\bm{\xi}}[\nabla\phi^{n}(\bm{\xi})].
\end{equation}
No further approximation is needed for $r=V_{\xi}^{n}/V_{x}^{n}$
since the old and new cell volumes can be computed explicitly at every
iteration. The linearisation will now be less accurate than in the
Euclidean case due to the curvature of the sphere, so it may be necessary
to increase $\alpha$ further to avoid divergence. 

\marginpar{Rev3.4 Ed2} \sout{The fixed-point iteration sequences
defined in eqns. (\mbox{\ref{eq:PMA_SI}}) and (\mbox{\ref{eq:PMA_SIsph}})
are preferred to the solution of the parabolic Monge-Amp\`ere equation
as used by \mbox{\citet{BCW13,BBPC14}} since the Laplacian term used
above is a linearisation of the Hessian and so will accelerate convergence
to the given monitor function whereas the Laplacian used by \mbox{\citet{BCW13,BBPC14}}
smoothes the solution towards a uniform mesh. However there may be
further scope for improvement by combining the best aspects of the
two approaches.}

\subsection{Spatial Discretisation on the Computational Mesh\protect\marginpar{Ed5}}

\marginpar{Ed2} \sout{For the solution of the Monge-Amp\`ere
equation in Euclidian geometry to be convex and converge monotonically,
\mbox{\citet{FO11}} use a Newton solver and a wide-stencil finite-difference
scheme. In Euclidean geometry we are solving a different version of
the Monge-Amp\`ere eqn. (\mbox{\ref{eq:PMA_SI}}) and on the sphere
(eqn \mbox{\ref{eq:PMA_SIsph}}), no previous solution technique exists.
We use fixed-point iterations and} We have found by experience, trial
and error and by analogy with numerical solution techniques for the
rotating shallow-water equations \citep[eg ][]{TRSK09}, some desirable
properties of the spatial discretisation in order to achieve convergence.
Further work to improve the spatial discretisation and prove convergence
is needed. 
\begin{enumerate}
\item The discretisation of $|I+H\left(\phi^{n}\right)|$ should be consistent
with the discretisation of $1+\nabla^{2}\phi$ otherwise the linearisation
will not be close and the iterative solution will not converge quickly.
In this context, consistent means that the trace of the discretised
$H\left(\phi\right)$ must be equal to $\nabla^{2}\phi$, as occurs
analytically. This is only possible when solving eqn (\ref{eq:PMA_SI}),
not eqn. (\ref{eq:PMA_SIsph}) since the relationship between $r$
and $1+\nabla^{2}\phi$ is not known numerically.
\item The spatial discretisation should be at least second-order accurate
and the errors should be smooth. If we have rough truncation errors
or first-order accurate truncation errors then truncation errors could
lead to mesh tangling.\marginpar{Ed5}
\item To avoid grid-scale oscillations in the solution of $\phi$, the spatial
discretisation should be as compact as possible so that grid-scale
oscillations of $\phi$ are not hidden in the discretisations of $|I+H\left(\phi^{n}\right)|$
and $m\left(\mathbf{x}\right)$.
\item If the solution, $\phi$, is convex or locally convex, then convex
cells in the initial mesh should remain convex in the mapped mesh.
This implies that $\nabla\phi$ should have bounded variation. %
\footnote{We postulate that, following the TRiSK discretisation on polygons
\citep{TRSK09}, the divergence of the mesh map on the initial (primal)
mesh should be a convex combination of the divergence of the mesh
map if it were calculated on a dual mesh (eg a triangulation). %
}
\end{enumerate}
We are considering a finite-volume discretisation on initially orthogonal
nearly uniform polygonal prisms. \marginpar{Rev2.3}The discretisation
that we describe is defined for arbitrary two-dimensional orthogonal
meshes consisting of shapes with any number of sides. This and the
above requirements suggests the following spatial discretisation on
a fixed\marginpar{Rev2.4} computational mesh:

\subsubsection{Discretisation of the Laplacian\label{sub:discLaplace}}

For cell $i$ with faces $f\in i$, the simplest, most compact discretisation
of the Laplacian, suitable for an orthogonal grid, using Gauss's divergence
theorem, is:\marginpar{Rev3.6}
\begin{equation}
\nabla_{i}^{2}\phi\approx\frac{1}{V_{i}}\sum_{f\in i}\nabla_{nf}\phi|\mathbf{S}_{f}|
\end{equation}
where cell $i$ has volume $V_{i}$, $\mathbf{S}_{f}$ is the outward
pointing normal vector to cell $i$ at face $f$ with area equal to
the face area so that $|\mathbf{S}_{f}|$ is the face area and gradient
normal to\marginpar{Ed5} each face is:
\begin{equation}
\nabla_{nf}\phi=\frac{\phi_{i_{f}}-\phi_{i}}{|\mathbf{d}_{f}|}\label{eq:snGrad}
\end{equation}
where cell $i_{f}$ is the cell on the other side of face $f$ from
cell $i$ and $|\mathbf{d}_{f}|$ is the (geodesic) distance between
cell centre $i$ and $i_{f}$ in the computational domain\marginpar{Ed5}.
This simple form ensures curl free pressure gradients (assuming that
the curl is calculated using Stokes circulation theorem around every
edge of the 3D mesh). If cell $i$ has centre $\bm{\xi}_{i}$ then
$\mathbf{d}_{f}=\bm{\xi}_{i}-\bm{\xi}_{i_{f}}$ in Euclidean geometry.
On the surface of the sphere, $|\mathbf{d}_{f}|$ is the great circle
distance between $\bm{\xi}_{i}$ and $\bm{\xi}_{if}$. Locations $\bm{\xi}_{i}$
and $\bm{\xi}_{i_{f}}$, vector $\mathbf{S}_{f}$ and $\mathbf{d}_{f}$
for cells $i$ and $i_{f}$ are shown in fig. \ref{fig:cell}(a). 

\begin{figure}
\begin{centering}
\includegraphics{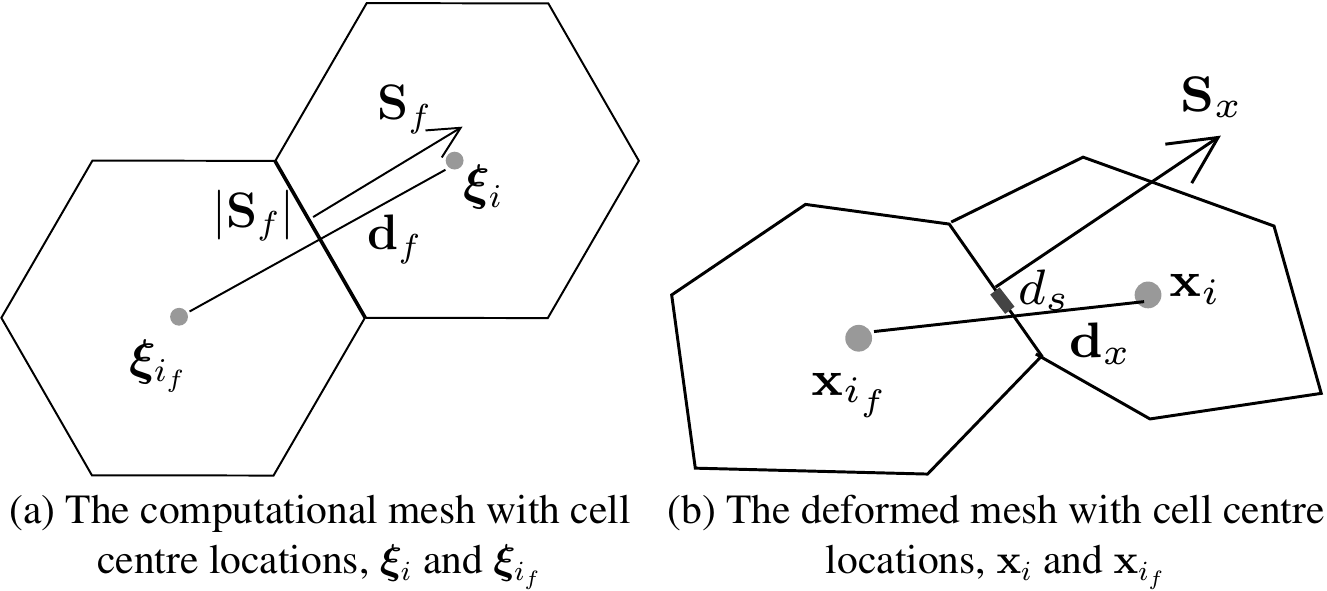}
\par\end{centering}

\caption{Cells $i$ and $i_{f}$ of the computational and deformed meshes either
side of face $f$, face area vector $\mathbf{S}_{f}$ and vector between
cell centres, $\mathbf{d}_{f}$. The skewness of face $f$ of the
deformed mesh (b) is shown by a short grey line of length $d_{s}$
in the plane of the face (see section \ref{sec:diags}). \label{fig:cell}}
\end{figure}

\subsubsection{Discretisation of the Hessian}

Two approaches are taken to calculate the Hessian. The first we define
a finite-difference approach (which uses both finite volume and finite
difference approximations). The second uses the fact that, in solving
the Monge-Amp\`ere equation for mesh generation, we are approximating
the change in cell volume by the determinant of the Hessian. Therefore,
rather than calculating a discretised Hessian, we can simply use the
change in cell volume, $r$. This is the geometric approach. The geometric
approach is always used on the surface of the sphere.\marginpar{Rev1.1}

\paragraph{Finite Difference Discretisation of the Hessian}

For a discretisation of the Hessian consistent with the discretisation
of the Laplacian, we use Gauss's theorem:\marginpar{Rev3.6}
\begin{equation}
H\left(\phi\right)_{i}=\nabla\nabla_{i}\phi=\frac{1}{V_{i}}\sum_{f\in i}\nabla_{f}\phi\mathbf{S}_{f}\label{eq:FDhess}
\end{equation}
where $\nabla_{f}\phi$ is the vector gradient of $\phi$ located
at face $f$ of cell $i$. The vector gradient, $\nabla_{f}\phi$,
is reconstructed from normal components, $\nabla_{nf}\phi$ using
a least-squares fit which is derived by assuming that $\nabla\phi$
is uniform so that it is first-order accurate on non-uniform meshes.
This approach starts by reconstructing a cell-centred gradient from
surrounding normal gradients using a least squares fit:
\begin{equation}
\nabla_{i}\phi=\left(\sum_{f\in i}\hat{\mathbf{S}}_{f}\mathbf{S}_{f}^{T}\right)^{-1}\sum_{f\in i}\nabla_{nf}\phi\ \mathbf{S}_{f}.\label{eq:gradc}
\end{equation}
\marginpar{Ed5}where $\hat{\mathbf{S}}_{f}=\mathbf{S}_{f}/|\mathbf{S}_{f}|$.
Next, a temporary value of the vector valued gradient at each face
is calculated:
\begin{equation}
\nabla_{f}^{\prime}\phi=\lambda_{f}\nabla_{i}\phi+\left(1-\lambda_{f}\right)\nabla_{i_{f}}\phi
\end{equation}
where $\lambda_{f}$ is the coefficient for linear interpolation.
For consistency with the Laplacian, we must have $\nabla_{f}\phi\cdot\mathbf{S_{f}}=\nabla_{nf}\phi|\mathbf{S}_{f}|$
which can be enforced with an explicit correction:
\begin{equation}
\nabla_{f}\phi=\nabla_{f}^{\prime}\phi+\left(\nabla_{nf}\phi-\nabla_{f}^{\prime}\phi\cdot\hat{\mathbf{S}}_{f}\right)\hat{\mathbf{S}}_{f}.\label{eq:gradf}
\end{equation}
The Hessian calculated using eqn. (\ref{eq:FDhess}) is not symmetric,
as the analytic version would be.

\paragraph{Geometric approach to calculating the Hessian}

A numerical approximation for calculating $H$ will introduce truncation
errors so instead we can simply use the change in cell volume:
\begin{equation}
r_{i}=|I+H_{i}\left(\phi\right)|=\frac{V_{i}\left(\mathbf{x}\right)}{V_{i}\left(\bm{\xi}\right)}
\end{equation}
where $V_{i}\left(\mathbf{x}\right)$ is the volume of the transported
mesh cell $i$ and $V_{i}\left(\bm{\xi}\right)$ is the volume of
the original cell. Volumes are calculated on the surface of the sphere
by decomposing every polyhedron (on the original and new meshes) into
tetrahedra with curved surfaces which are flat in spherical geometry.
The volumes of these tetrahedra are found using the formula for the
area of a spherical triangle.

\subsubsection{The Gradient at the Vertices}

In order to calculate the mesh map and consequently to calculate $m$,
we must calculate $\nabla\phi$ at the mesh vertices, $\nabla_{v}\phi$.
(This is in contrast to \marginpar{Rev1.4}\citealp{FroeseThesis2012,SAK1x}
who discretise the gradient of $\phi$ at the same locations where
$\phi$ is stored.) Ideally, the calculation of $\nabla_{v}\phi$
should not produce any non-convex cells and it turns out to be particularly
sensitive to the numerical approximation and its stencil. In section
\ref{par:wideGradv} we will describe a large stencil gradient, for
which grid-scale oscillations in $\phi$ can grow which are not seen
in $\nabla_{v}\phi$ and convergence is slow. \marginpar{Rev3.7}In
section \ref{par:smallGradv} we will describe a small stencil gradient
which can lead to grid-scale oscillations of $\nabla_{v}\phi$ on
a hexagonal mesh since the calculation of the gradient does not lead
to a gradient with bounded variation. This leads to locally distorted
meshes. \marginpar{Rev2.6}Section \ref{par:goldilocksGradv} describes
a compromise; a new Goldilocks stencil that combines the advantages
of both large and small.

\paragraph{Vertex Gradient using a Large Stencil\label{par:wideGradv}}

In order to calculate $\nabla_{v}\phi$ at the vertices, $\nabla_{f}\phi$
at the faces is calculated using eqn. (\ref{eq:gradf}). These values
are then interpolated onto the vertices using linear interpolation.
On a mesh of squares, four values of $\nabla_{f}\phi$ are averaged
to calculate each $\nabla_{v}\phi$ at a vertex and on a mesh of hexagons,
three values of $\nabla_{f}\phi$ are averaged to calculate each $\nabla_{v}\phi$.
Including the calculation of $\nabla_{f}\phi$, the reconstruction
of $\nabla_{v}\phi$ from $\phi$ uses a stencil of 10 hexagons on
a hexagonal mesh and 12 squares on a mesh of squares. Due to these
large stencils, the gradients calculated are smooth even if the $\phi$
field is not smooth. Due to the averaging (interpolation) of the gradient
from the cell centres to the vertices there will be some consistency
between gradients at different vertices and so cells may remain convex.

\paragraph{Vertex Gradient using a Small Stencil\label{par:smallGradv}}

The vector gradient at each vertex, $\nabla_{v}\phi$, can be reconstructed
directly from the normal component of the gradient at the surrounding
faces using a least squares fit
\begin{equation}
\nabla_{v}\phi=\left(\sum_{f\in v}\mathbf{d}_{f}\mathbf{d}_{f}^{T}\right)^{-1}\sum_{f\in v}\left(\mathbf{d}_{f}\nabla_{nf}\phi\right)\label{eq:compactGradv}
\end{equation}
where $f\in v$ is the set of faces which share vertex $v$. This
approximation is exact for a uniform vector field, $\nabla\phi$,
and is consequently first order accurate on an arbitrary mesh. However
on a hexagonal mesh, eqn. (\ref{eq:compactGradv}) only uses information
from three surrounding faces and three surrounding hexagons and the
resulting gradients are prone to grid-scale oscillations which can
lead to the creation of non-convex cells. The small amount of information
used at every vertex means that neighbouring vertices can have very
different gradients. We therefore need a larger stencil, but not as
large as the stencil used in section \ref{par:wideGradv}. 

On a mesh of squares, $\phi$ at four squares is sufficient to reconstruct
a smooth $\nabla_{v}\phi$ to second order.

\paragraph{Vertex Gradient using the Goldilocks Stencil\label{par:goldilocksGradv}}

The Goldilocks stencil should be large enough to calculate a smooth
gradient (with bounded variation) but without including averaging
which can hide grid-scale oscillations in $\phi$. The stencil used
includes the faces which share vertex $v$ and the face neighbours
attached by a vertex to those faces (fig. \ref{fig:goldilocksStencil}).
The vertex gradient is then reconstructed using a least squares fit:
\begin{equation}
\nabla_{v}\phi=\left(\sum_{f\in v^{\prime}\in f^{\prime}\in v}\mathbf{d}_{f}\mathbf{d}_{f}^{T}\right)^{-1}\sum_{f\in v^{\prime}\in f^{\prime}\in v}\left(\mathbf{d}_{f}\nabla_{nf}\phi\right)\label{eq:golilocksGradv}
\end{equation}
\marginpar{Rev2.10b}where $f\in v^{\prime}\in f^{\prime}\in v$ is
the set of faces shown\marginpar{Rev3.8} \sout{by dashed lines}
in fig. \ref{fig:goldilocksStencil}. In the least squares fit in
eqn. (\ref{eq:golilocksGradv}), the central faces are counted three
times (making the fit more accurate near the centre, following \citet{WWF09}). 

\begin{figure}
\begin{centering}
\includegraphics{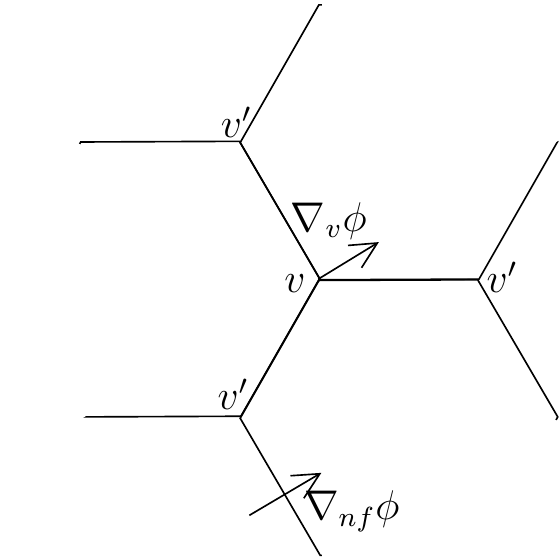}
\par\end{centering}

\caption{%
The Goldilocks stencil for recontructing the full gradient of the
mesh potential, $\nabla_{v}\phi$, at the central vertex, $v$, from
the component of the normal gradient at each face, $\nabla_{nf}\phi$.
This results in a stencil of six hexagons to calculate $\nabla_{v}\phi$
at a vertex.\label{fig:goldilocksStencil}%
}
\end{figure}

\subsubsection{Calculation of Exponential Maps\protect\marginpar{Rev1.6}}

Exponential maps are used to move vertices on the surface of the sphere.
The direction of the map is given by the direction of $\nabla_{v}\phi$
at vertex $v$ (ie the direction is along the great circle in the
plane of $\nabla_{v}\phi$) The distance moved is the geodesic distance
$|\nabla_{v}\phi|$ so that the vertex is rotated around the sphere
by an angle $|\nabla_{v}\phi|/a$ where $a$ is the radius of the
sphere.

\subsubsection{Linear equation solver and fixed point iterations}

\marginpar{Rev3.9}Spatial discretisation of eqns (\ref{eq:PMA_SI}),(\ref{eq:PMA_SIsph})
leads to a set of linear algebraic equations, which can be written
as a matrix equation, $A\bm{\phi}^{(n+1)}=\mathbf{b}^{(n)}$, where
$\bm{\phi}^{(n+1)}$ is the vector of all of the values of the unknown,
$\phi^{(n+1)}$. This matrix equation is solved using the OpenFOAM
GAMG solver (geometric algebraic multi-grid, \citep{OpenFOAM}) using
diagonal incomplete Cholesky smoothing with 50 cells in the coarsest
level. The residual for the solver tolerance is defined as:
\begin{equation}
\frac{\sum|\mathbf{b}-A\bm{\phi}|}{\sum\left(|\mathbf{b}|+|A\mathbf{\bm{\phi}}|\right)}\label{eq:initResid}
\end{equation}
where the sum is over all cells of the mesh (ie over all elements
of the vectors $\mathbf{b}$ and $A\bm{\phi}$). \marginpar{Rev3.11}For
each fixed-point iteration, the values of $\phi$ from the previous
iteration are used as an initial guess for the solution of the matrix
equation, so the initial residual should converge to the final residual
as the fixed-point iterations converge. At each fixed point iteration
(ie each value of $n$ in eqn. (\ref{eq:PMA_SI})) the matrix equation
is solved with a tolerance equal to the maximum of 0.001 times the
initial residual and $10^{-8}$. The matrix equation is not solved
all the way to $10^{-8}$ at every fixed-point iteration to save computational
cost but, when the fixed-point iterations have converged, the initial
residual will be less than $10^{-8}$. A weaker tolerance is probably
acceptable for mesh generation but we are using a tight tolerance
to have more confidence that the numerical method is convergent.

\subsubsection{Moving Voronoi Generating Points}

If the initial mesh is Voronoi and it is required that the transported
mesh is also Voronoi, then the Voronoi generating points can be moved
using eqn. (\ref{eq:meshMap}) using the cell centre gradient, reconstructed
from the face gradient using volume weighting. Then the moved generating
points can be re-tesselated to create a new Voronoi tesselation. However,
the re-tesselation may not have exactly the same connectivity due
to edge swapping in the Delaunay algorithm. This technique therefore
may not be so suitable for r-adaptivity.

\subsubsection{Calculating the Monitor Function}

When using r-adaptivity, the mesh monitor function (that controls
the mesh density) will need to be mapped from the previous mesh onto
the new transported mesh so that it can be evaluated when solving
eqns. (\ref{eq:PMA_SI}) or (\ref{eq:PMA_SIsph}). In section \ref{sec:results},
we first present results using an analytic monitor function, \marginpar{Rev3.10}which
is evaluated at the transported mesh cell centres. We then \marginpar{Rev3.10}use
observed meteorological data to calculate a monitor function by mapping
the data to the computational grid and then \marginpar{Rev3.10}apply
Laplacian smoothing, as described in section \ref{sub:smoothingm}.

\subsection{Enforcing Stability }

\subsubsection{Under-relaxation\label{sub:underRelax}}

Here we describe how $\alpha$ is calculated. We start by defining
the source terms of the eqns (\ref{eq:PMA_SI}),(\ref{eq:PMA_SIsph})
to be $s^{n}=|I+H\left(\phi^{n}\right)|-c/m({\bf x}^{n})$ and $s^{n}=r\left(\phi^{n}\right)-c/m({\bf x}^{n})$
respectively. For convergence to occur, we would like the source term
to decrease relative to the Laplacian term, $\nabla^{2}\phi$. Initially,
the source term has order 1. In the tests undertaken, both on the
plane and on a sphere, it has been sufficient to keep the ratio of
the Laplacian to the source term greater than four and to always ensure
that $\alpha$ increases with iteration number, $n$. So $\alpha$
is set to be:
\begin{equation}
1+\alpha^{n+1}=\max\left(1+\alpha^{n},\ 4\max\left(^{1}\!\!/\!_{4},\ \max\left(|s^{n}|\right)\right)\right)
\end{equation}

\subsubsection{Smoothing the Monitor Function\label{sub:smoothingm}}

Following \citet{BBPC14}, we experimented with smoothing the monitor
function and this smoothing certainly improved convergence and generated
meshes with smoother grading and hence lower anisotropy and skewness
and better orthogonality (see section \ref{sec:diags}). However the
purpose of this work is to describe a robust solution of the Monge-Amp\`ere
equation on the sphere for any monitor function. So smoothing of the
monitor function will not be considered for the analytically defined
monitor functions. However, when using meteorological data to define
a monitor function, the monitor function is smoothed on the computational
grid during each iteration using Laplacian smoothing:

\begin{eqnarray*}
m & = & m^{\prime}\left(\exp_{\xi}\nabla\phi\right)+\frac{1}{4}\nabla\cdot\left(|\mathbf{d}_{f}|^{2}\nabla m^{\prime}\left(\exp_{\xi}\nabla\phi\right)\right)
\end{eqnarray*}
where $m^{\prime}$ is the monitor function mapped from the meteorological
data onto the physical grid at position $\exp_{\xi}\nabla\phi$ and
$m$ is the monitor function used in the source terms of eqns. (\ref{eq:PMA_SI}),(\ref{eq:PMA_SIsph}).
The diffusion coefficient used is the square of the mesh spacing,
$|\mathbf{d}_{f}|$ on the computational grid.

\section{Diagnostics of Convergence and of Mesh Quality\label{sec:diags}}

\marginpar{Rev3.9}Convergence is measured in two ways. Firstly, convergence
is measured by plotting the initial residual (eqn \ref{eq:initResid})
of the matrix equation at every fixed point iteration as a function
of iteration number. This gives an indication of how much the solution
is changing for each fixed-point iteration.

Secondly, the convergence of the final solution is assessed by plotting
the change in cell area between the initial mesh and the final iteration
for every cell in comparison to $c/m$. At convergence, these should
be equal. The test cases considered use an axi-symmetric monitor function,
$m$, so this measure is plotted as a scatter plot against distance
to the axis of symmetry. This tells us where the solution is not converging
to the required monitor function and also, for solutions using $|I+H\left(\phi^{n}\right)|$
instead of $r\left(\phi^{n}\right)$ (ie using the finite difference
Hessian rather than the geometric Hessian) in the Monge-Amp\`ere
equation, it tells us how well $|I+H\left(\phi^{n}\right)|$ approximates
$r\left(\phi^{n}\right)$.

The diagnostics of mesh quality consider cell centres, defined as
cell centroids or centres of mass of the moved cells, and face centres,
defined in the same way. \marginpar{Rev3.12}These will also use the
face area vector, $\mathbf{S}_{x}$, the normal vector to each face
with magnitude equal to the face area and $\mathbf{d}_{x}$, the vector
between cell centres either side of a face of the deformed mesh (see
fig \ref{fig:cell}(b)).

In order to measure mesh quality, firstly we will consider mesh spacing,
$|\mathbf{d}_{x}|$, between adjacent cell centres for each cell face
as a scatter diagram as a function of distance to the axis of symmetry
(for the axi-symmetric cases). This informs us about the aspect ratios
of the cells since cells with high aspect ratio will give a large
scatter of values of $|\mathbf{d}_{x}|$ for a given distance to the
axis of symmetry. If the mesh is perfectly equidistributed then cell
areas should be given by $c/m$. Therefore, for the meshes of quadrilaterals,
$|\mathbf{d}_{x}|$ will be compared with $\sqrt{c/m}$ and for the
meshes of mostly hexagons, $|\mathbf{d}_{x}|$ will be compared with
$\sqrt{2c\ \tan(\pi/3)/3m}$ .

The second mesh quality diagnostic is non-orthogonality for each cell
face which is measured as
\begin{equation}
\text{non-orthogonality}=\cos^{-1}\frac{\mathbf{S}_{x}\cdot\mathbf{d}_{x}}{|\mathbf{S}_{x}||\mathbf{d}_{x}|}.\label{eq:nonOrthog}
\end{equation}

The third mesh quality diagnostic is the face skewness, measured as
the distance, $d_{s}$, between the face centre and the crossing point
between the vector $\mathbf{d}_{x}$ with the face, normalised by
$|\mathbf{d}_{x}|$:
\begin{equation}
\text{skewness}=\frac{d_{s}}{|\mathbf{d}_{x}|}.\label{eq:skewness}
\end{equation}
The skewness distance, $d_{s}$, is shown as a short grey line in
fig. \ref{fig:cell}(b). This definition of skewness is a feature
of the non-linearities of the map generating the mesh and is different
quantitatively and qualitatively from that of \citet{BRW15} which
can be calculated directly the Jacobian of the map. The skewness metric,
$Q$, from \citet{BRW15} gives information about isotropy and orthogonality,
not face skewness.

\section{Results\label{sec:results}}

Optimally transported meshes are generated in two-dimensional planar
geometry to compare with those generated by numerical solution of
the parabolic Monge-Amp\`ere equation by \citet{BRW15}. Next, OT
meshes are generated on the surface of the sphere in order to compare
with the centroidal Voronoi meshes generated by \citet{RJG+11} using
Lloyd's algorithm. Finally, OT meshes are generated on the sphere
using observed precipitation to define a monitor function.

\subsection{Optimally Transported Meshes in Euclidean Geometry}

Meshes are generated on a finite plane, $[-1,1]^{2}$, using the radially
symmetric monitor function used by \citet{BRW15} defined for each
location $\mathbf{x}_{i}$:
\begin{equation}
m\left(\mathbf{x}_{i}\right)=1+\alpha_{1}\text{sech}^{2}\left(\alpha_{2}\left(R^{2}-a^{2}\right)\right)\label{eq:planeMonitor}
\end{equation}
where $\mathbf{x}_{c}$ is the centre of the refined region (the origin
for these results), $R$ is the distance of $\mathbf{x}_{i}$ to $\mathbf{x}_{c}$
and $\alpha_{1}$, $\alpha_{2}$ and $a$ control the variations of
the density function. Following \citet{BRW15} we generate two types
of mesh with this monitor function, the first we call the ring mesh
using $a=0.25$, $\alpha_{1}=10$ and $\alpha_{2}=200$ and the second
the bell mesh using $a=0$, $\alpha_{1}=50$ and $\alpha_{2}=100$,
both using periodic boundary conditions for $\phi$. \marginpar{Rev2.3}
The computational meshes on which the optimal transport problems are
solved are uniform grids of $60\times60$ squares.\marginpar{Rev2.7}

The ring and bell meshes generated using both the finite difference
and the geometric Hessian on the plane are shown in figure \ref{fig:planeMeshes}.
The convergence diagnostics will be presented in section \ref{sub:convPlane}.
Mesh quality for these meshes was analysed by \citet{BRW15} and this
is not repeated here.

\begin{figure}
\includegraphics[width=1\linewidth]{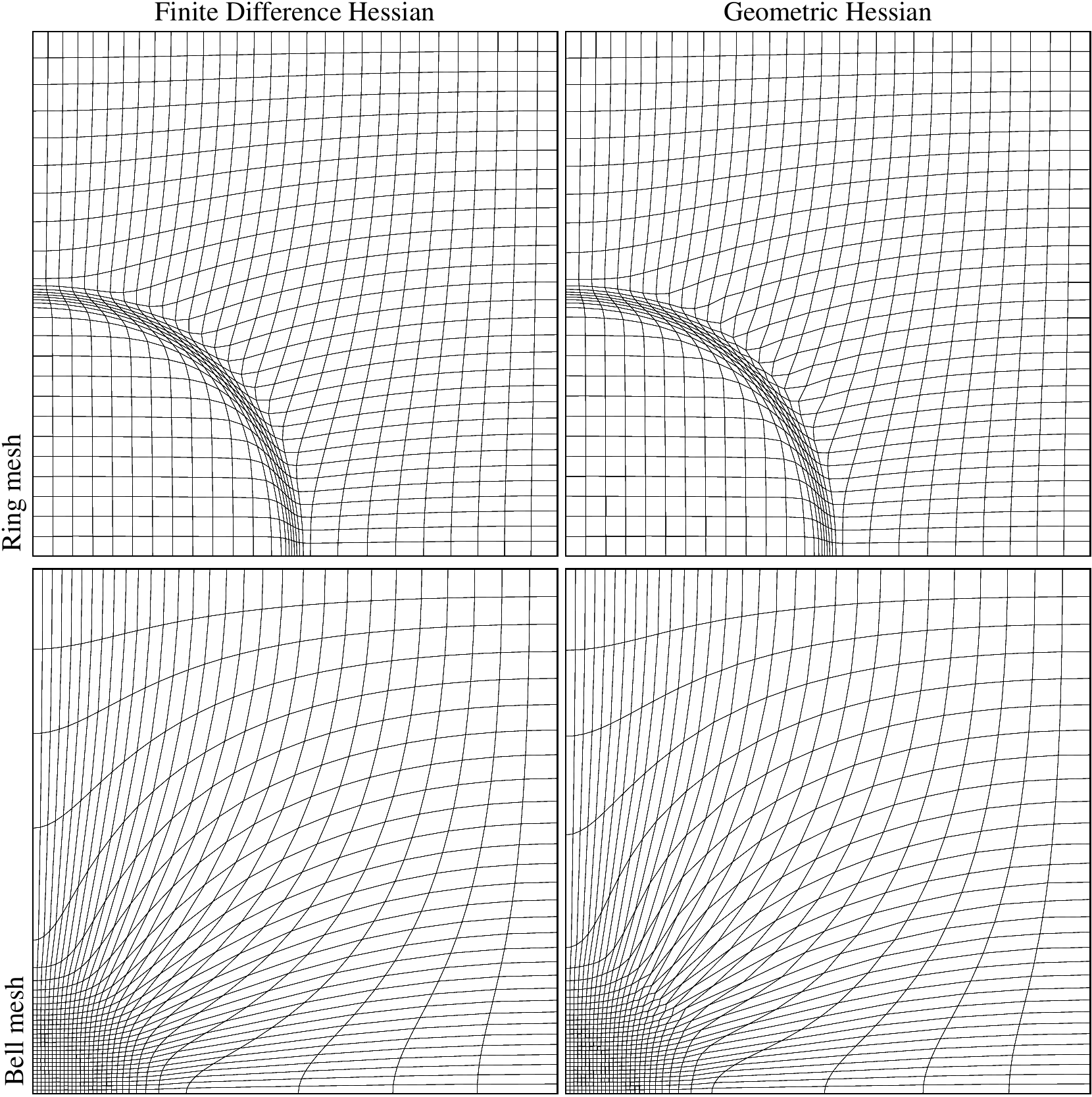}

\caption{%
A quarter of the ring and bell meshes using the finite difference
and volume Hesssian. \label{fig:planeMeshes}%
}
\end{figure}

The meshes in fig. \ref{fig:planeMeshes} calculated using both Hessian
techniques are similar to each other and they are also similar to
the meshes generated by \citet{BRW15}.

\subsection{Convergence of the Monge-Amp\`ere Solution in Euclidean Geometry\label{sub:convPlane}}

\begin{figure}
\includegraphics[width=1\linewidth]{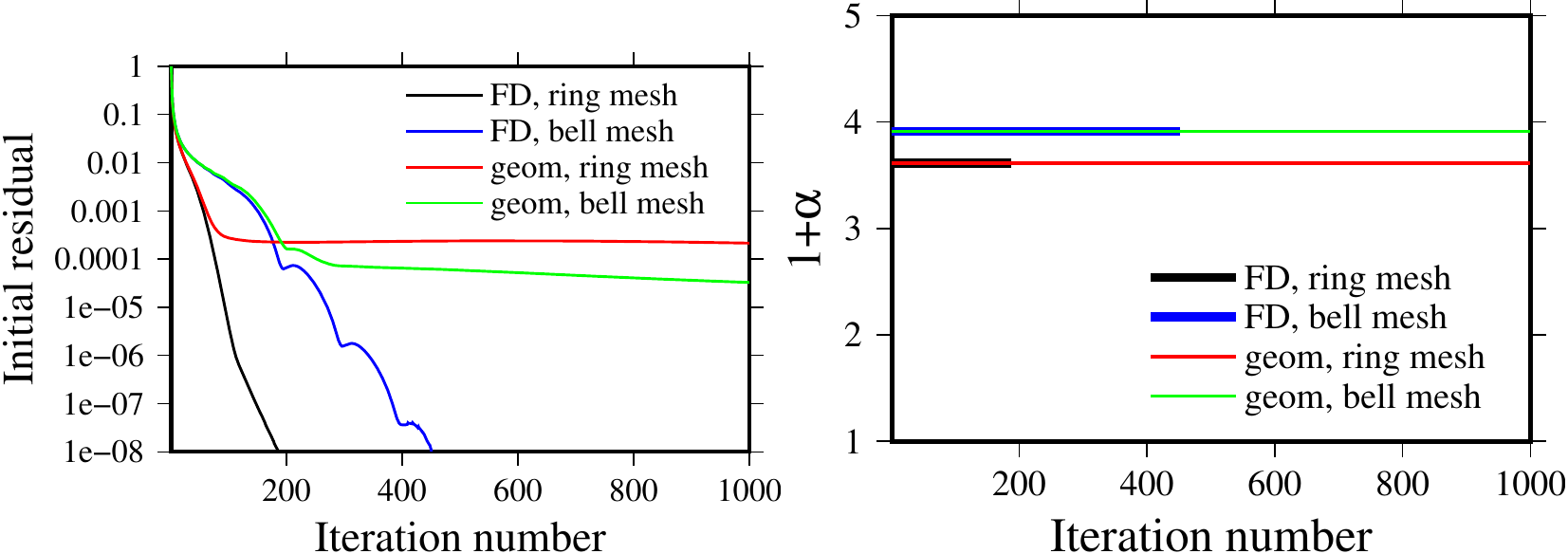}

\caption{%
Convergence of initial residual of the solution of the Poisson equation
for each fixed point iteration and $1+\alpha$ as a function of iteration
number for the planar meshes generated with the monitor function defined
in eqn. (\ref{eq:planeMonitor}). \label{fig:planeMeshConvergence}%
}
\end{figure}

Figure \ref{fig:planeMeshConvergence} on the left shows the initial
residual of the matrix solution as a function of iteration number
for the calculation of all of the meshes on the plane. Using the finite
difference Hessian, the solution converges rapidly but convergence
stalls when using the geometric Hessian. There are two possible reasons
for the stalling. Firstly, the Laplacian is no longer a good linearisation
of the geometric Hessian and secondly, a solution at this resolution
may not exist. Smoothing the monitor function removes the stall in
convergence and speeds convergence of all solutions (not shown) \marginpar{Rev2.8}since
smoothing removes the very abrupt changes in the monitor function.
However this is not the topic of this paper. 

The underelaxation factor, $1+\alpha$, is shown in the right of fig.
\ref{fig:planeMeshConvergence}. It never rises above the initial
value because the source term never increases above its initial value.
The initial value of $1+\alpha$ is simply $4\max|1-c/m|$ and so
$1+\alpha$ is independent of the Hessian calculation method. 

\begin{figure}
\includegraphics[width=1\linewidth]{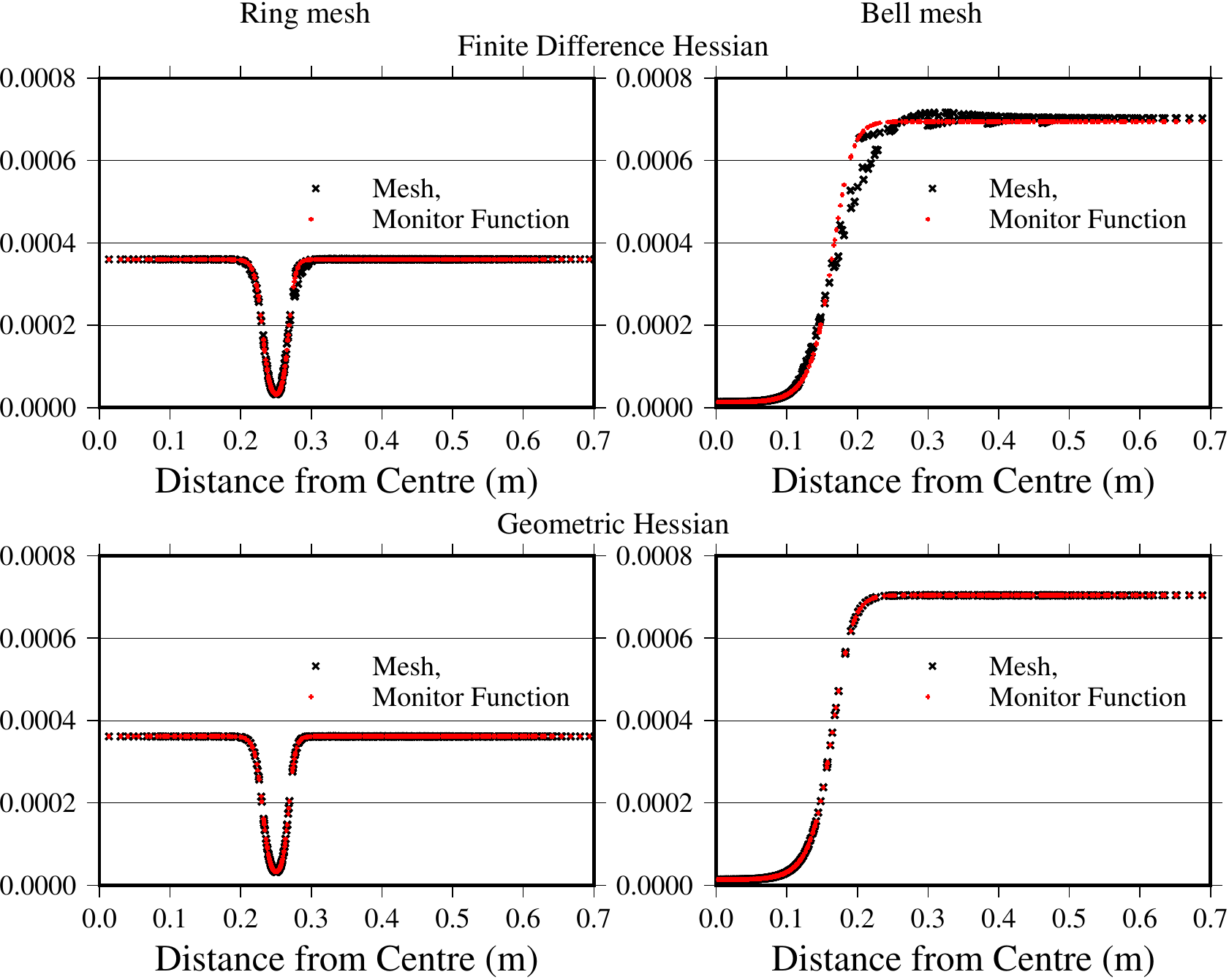}

\caption{%
Cell area as a function of distance from the axis of symmetry in comparison
to that implied by the monitor function ($c/m$) for the planar meshes
generated with the monitor function defined in eqn. (\ref{eq:planeMonitor}).
\label{fig:planeMeshArea}%
}
\end{figure}

In order to diagnose how closely the final mesh equidistributes the
monitor function, we plot the cell area as a scatter plot for every
cell in the mesh as a function of the distance from the axis of symmetry
in fig. \ref{fig:planeMeshArea} in comparison with $c/m$. Using
the finite difference Hessian, there are discrepancies between $c/m$
and the cell area where the second derivative of $c/m$ is high. This
is because the discrete calculation of the Hessian is not a good approximation
of the cell area in these regions, where the derivatives of $\phi$
are varying rapidly. However, for the purpose of mesh generation,
these discrepancies do not look problematic. If the OT mesh is smoother
than that specified by the monitor function then it could be beneficial.
The meshes generated using the geometric Hessian are more accurately
equidistributed with respect to the monitor function, despite the
lack of convergence of the initial residual.

\subsection{Optimally transported and Centroidal Voronoi meshes on the Sphere
\label{sub:sphMeshes}}

Meshes are generated using the geometric Hessian in order to compare
with the locally refined centroidal Voronoi meshes generated using
Lloyd's algorithm by \citet{RJG+11}. We use a monitor function given
by the square root of the density function of the corrected eqn. (4)
of \citet{RJG+11}:
\begin{equation}
m\left(\mathbf{x}_{i}\right)=\sqrt{\frac{1}{2\left(1+\gamma\right)}\left(\tanh\frac{\beta-||\mathbf{x}_{c}-\mathbf{x}_{i}||}{\alpha}+1\right)+\gamma}
\end{equation}
where $\mathbf{x}_{c}$ is the centre of the refined region which
has a latitude of $30^{o}$ and a longitude of $90^{o}$. $||\mathbf{x}_{c}-\mathbf{x}_{i}||$
is the geodesic distance between the points and is computed as $\cos^{-1}\left(\mathbf{x}_{c}\cdot\mathbf{x}_{i}\right)$.
$\alpha$ and $\beta$ are in radians and they control the size of
the refined region and the distance over which the mesh changes from
fine to coarse resolution. We follow \citet{RJG+11} and use $\alpha=\pi/20$
and $\beta=\pi/6$. $\gamma$ controls the ratio between the finest
and coarsest resolution and we use $\gamma=\left(1/2\right)^{4}$,
$\gamma=(1/4)^{4}$, $\gamma=(1/8)^{4}$ and $\gamma=(1/16)^{4}$
for meshes with finest mesh spacing factors of 2, 4, 8 and 16 times
smaller than that of the coarsest. Following \citet{RJG+11}, these
meshes are referred to as X2, X4, X8 and X16. 

\marginpar{Rev2.3} The computational meshes are hexagonal icosahedra
which consist of 12 pentagons and \mbox{$10(2^{2n}-1)$} hexagons
for $n=3,4,5,6$. \marginpar{Rev2.9}These quasi-uniform meshes can
be referred to as the X1 meshes. The X1 meshes are not shown but the
X1 centroidal Voronoi meshes and the OT meshes are slightly different.
The X1 centroidal Voronoi meshes are generated using Lloyd's algorithm
which guarantees that the X1 meshes are nearly centroidal (the Voronoi
generating point is co-located with the cell centre) whereas the X1
OT meshes are the \citet{HR95b} version of the hexagonal icosahedron,
optimised to reduce face skewness. The X2, X4, X8 and X16 meshes of
2,562 cells are shown in figure \ref{fig:X24meshes} with the ratio
between the cell area and the average cell area coloured. 

\begin{figure}
\includegraphics[width=1\linewidth]{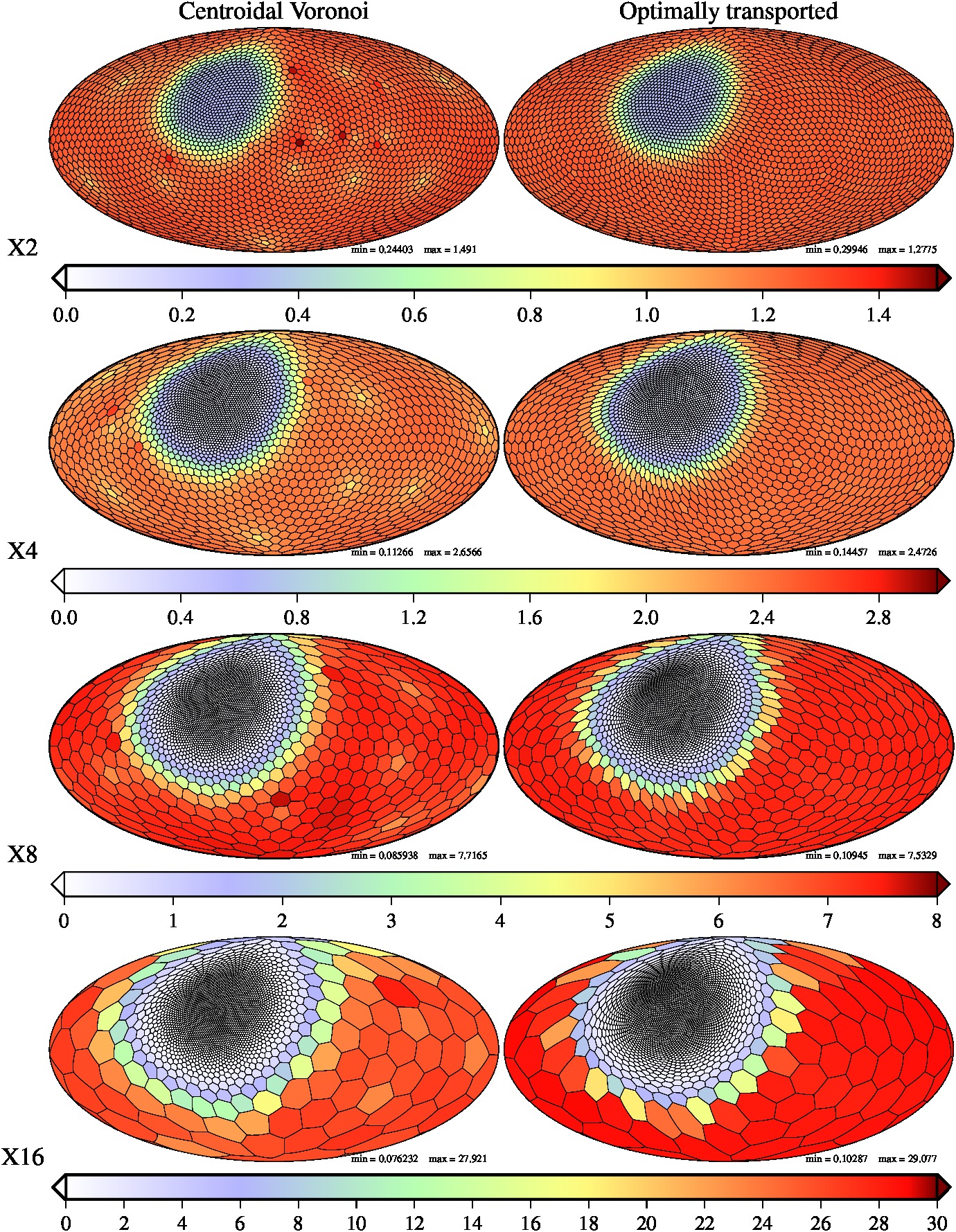}

\caption{%
X2, X4, X8 and X16 meshes of the sphere of 2,562 cells and the cell
volumes relative to the average cell volume in colour. The surface
of the entire sphere is displayed using a Hammer equal-area projection
\citep{WSS+13}. \label{fig:X24meshes}%
}
\end{figure}

The centroidal Voronoi meshes  in fig. \ref{fig:X24meshes} are orthogonal,
close to centroidal and the mesh topology (connectivity) is different
for all the refinement levels. (Lloyd's algorithm generates meshes
that are centroidal relative to a density function which means that
they are not centroidal when the centroid is simply the centre of
mass.) The OT meshes all have the same connectivity and they are centroidal
but not orthogonal. (Orthogonality could be achieved by Voronoi tesselating
the meshes, at the expense of centroidality, see section \ref{sub:betterMeshQuality}.)

All of the OT meshes in fig. \ref{fig:X24meshes} have regions of
anisotropy in between the fine and coarse regions whereas the centroidal
Voronoi meshes remain isotropic and the mesh topology changes between
resolutions. The anisotropy will be investigated further in section
\ref{sub:meshQualitySph}. Before looking in more detail at the mesh
quality in section \ref{sub:meshQualitySph}, we will examine diagnostics
of convergence in section \ref{sub:convSph}.

\subsection{Convergence of the Monge-Amp\`ere Solution on the Sphere\label{sub:convSph}}

\begin{figure}
\includegraphics[width=1\linewidth]{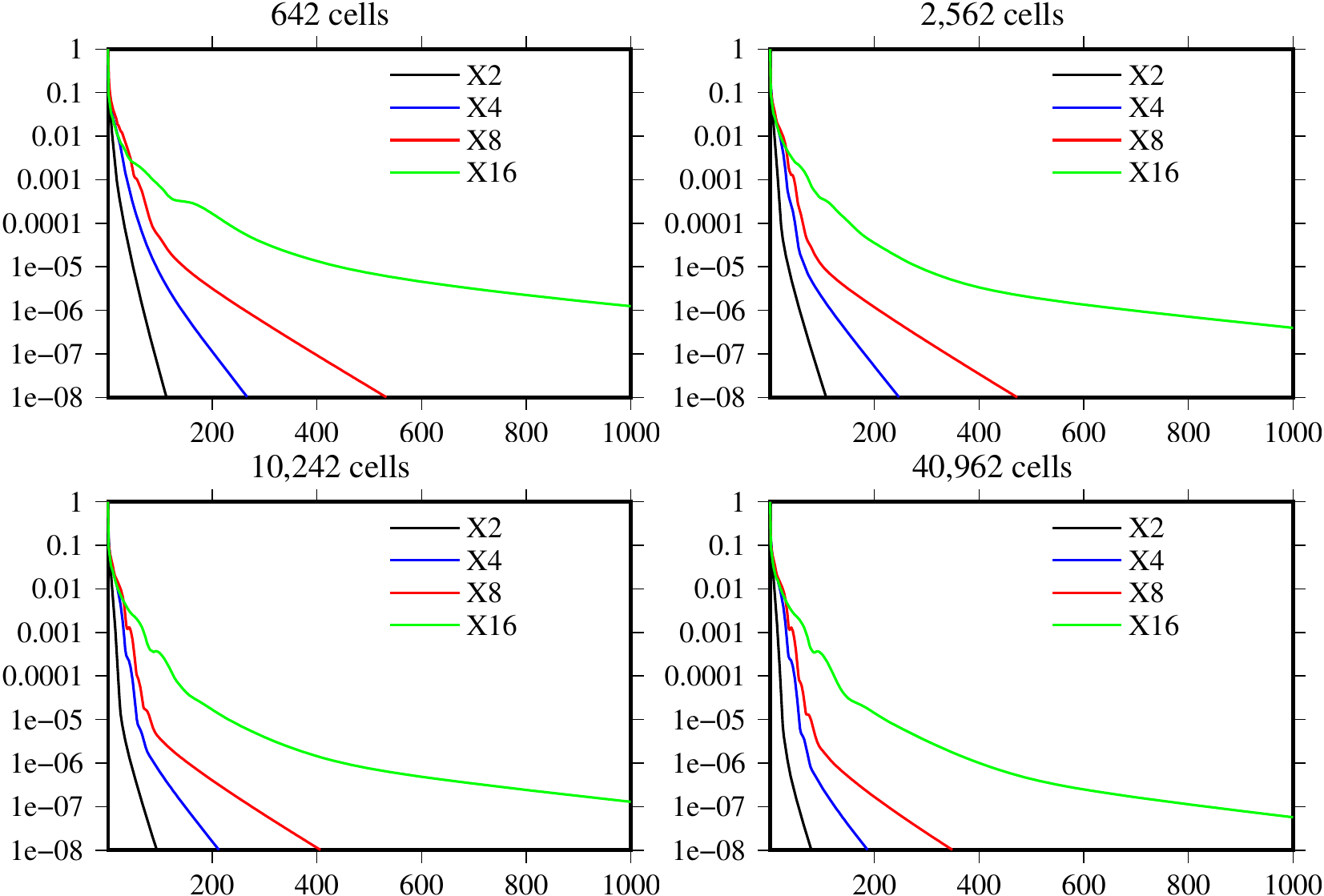}

\caption{Convergence of initial residuals with iteration number for all of
the optimally transported meshes on the sphere \label{fig:convergence}}
\end{figure}

Convergence of the initial residual is shown in fig. \ref{fig:convergence}
for the X2-X16 meshes of various resolutions. Convergence is rapid
for the X2 and X4 meshes but slows after around 100 iterations, once
the non-linearities have grown and the Laplacian is no longer a good
approximation for the Hessian and once the exact solutions becomes
difficult to achieve at finite resolution. It appears from fig. \ref{fig:convergence}
that the number of iterations reduces as mesh size increases.

\begin{figure}
\includegraphics[width=1\linewidth]{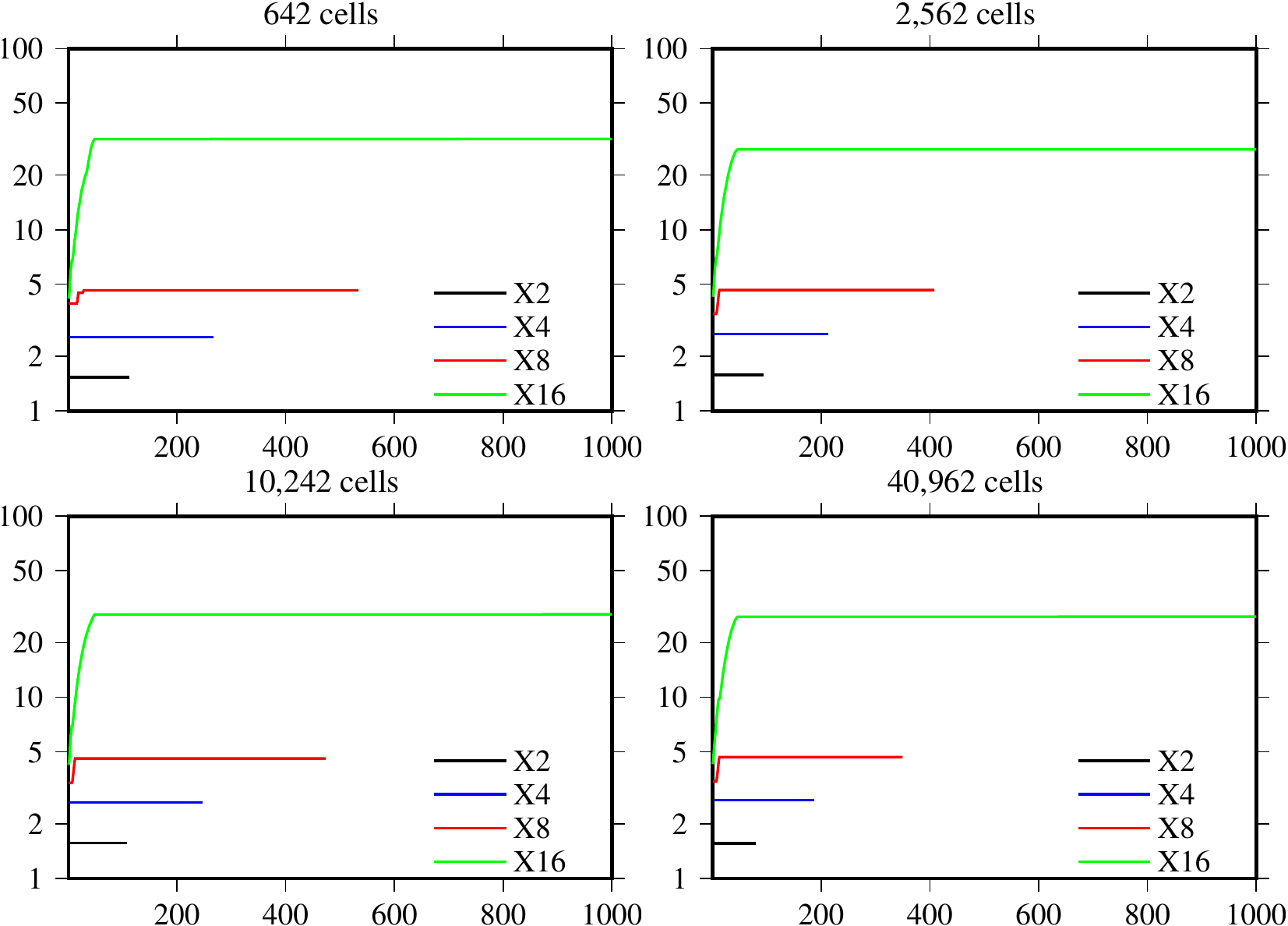}

\caption{Under-relaxation factor, $1+\alpha$, for with iteration number for
all of the optimally transported meshes on the sphere\label{fig:boostLaplacian}}
\end{figure}

The underelaxation factor, $1+\alpha$, is shown in fig. \ref{fig:boostLaplacian}.
Unlike in the Euclidean case, $1+\alpha$ does rise after initialisation.
This implies that the maximum of the source term increases before
it decreases. However the initial residual \marginpar{Rev3.13}is
monotonically decreasing during these early iterations. This is because
the initial residual is a mean over the whole domain whereas $1+\alpha$
is set from the maximum of the source term. 

\begin{figure}
\includegraphics[width=1\linewidth]{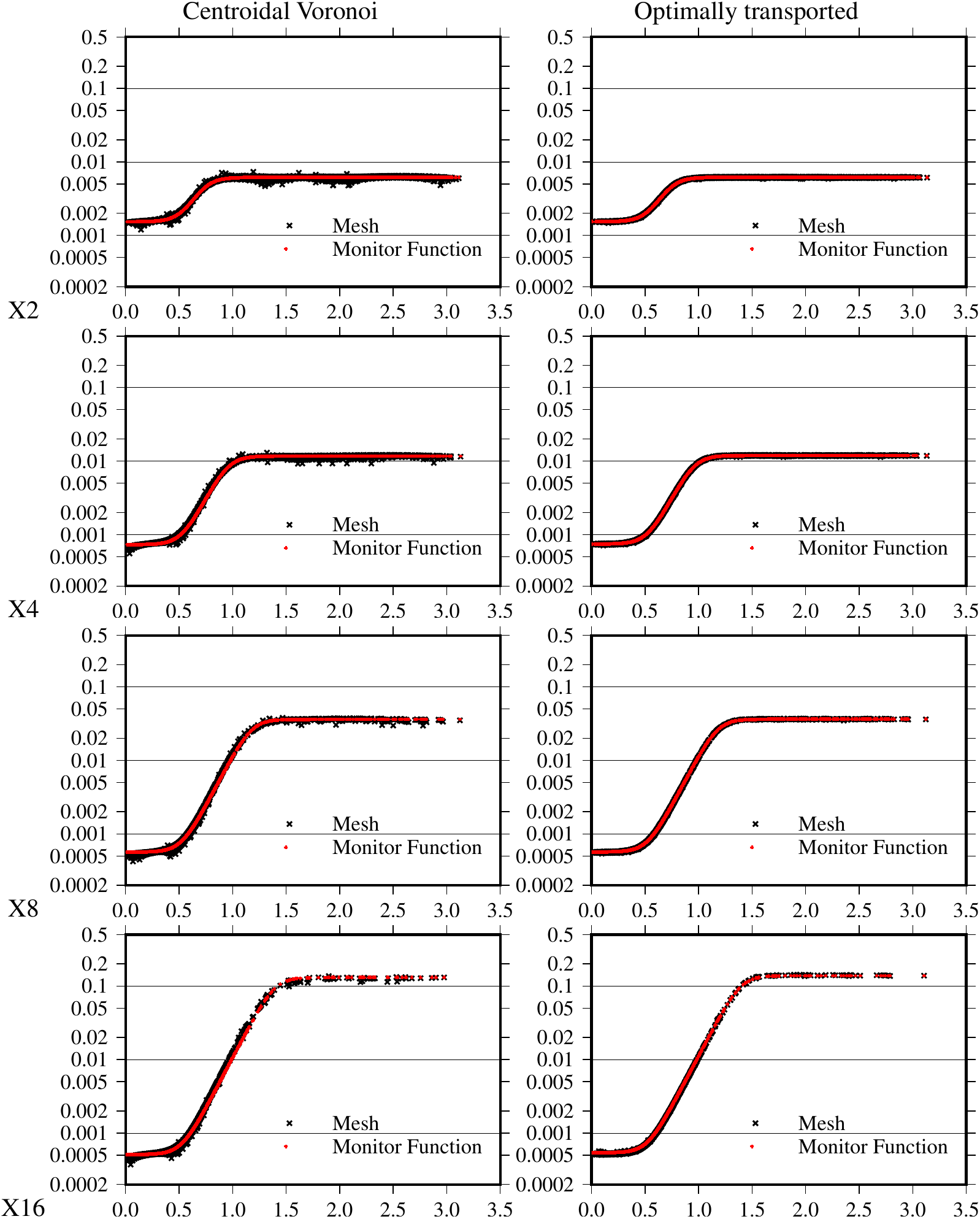}

\caption{Cell area as a function of geodesic distance to the centre of the
refined region in comparison to that implied by the monitor function
($c/m$) for the X2, X4, X8 and X16 meshes of 2,562 cells \label{fig:areaDxX2X4}}
\end{figure}

The convergence of the cell area with the monitor function is shown
in fig. \ref{fig:areaDxX2X4} as scatter plots of cell area change
as a function of distance to the centre of the refined region. As
occurred in Euclidean geometry, the geometric Hessian gives accurate
equidistribution.

\subsection{Mesh Quality on the Sphere\label{sub:meshQualitySph}}

\begin{figure}
\includegraphics[width=1\linewidth]{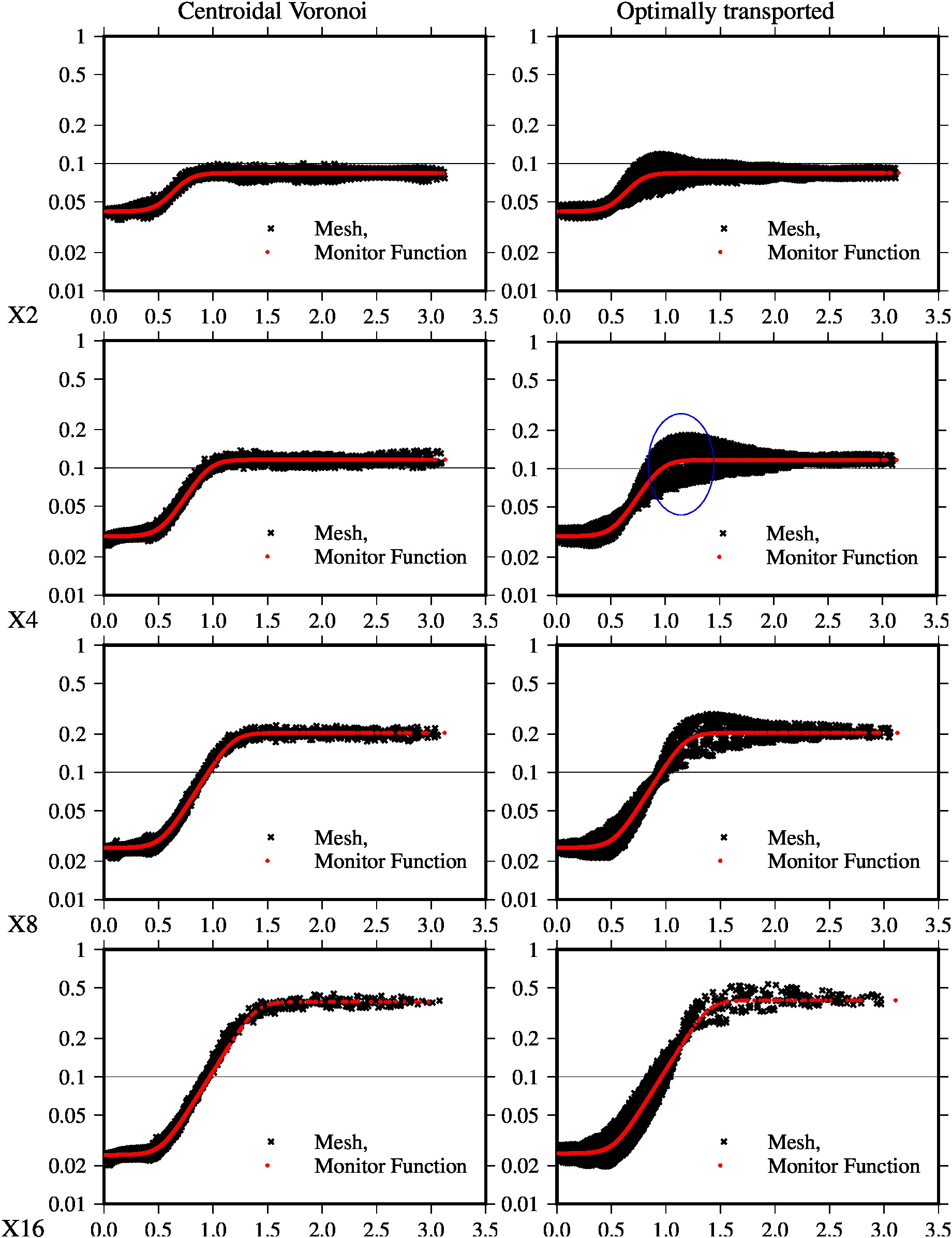}

\caption{Cell-centre to cell-centre geodesic distance ($|\mathbf{d}_{x}|$
as defined in section \ref{sec:diags}) as a function of geodesic
distance to the centre of the refined region in comparison to that
implied by the monitor function ($\sqrt{2c\ \tan(\pi/3)/3m}$) for
the X2, X4, X8 and X16 meshes of 2,562 cells. The blue encircled region
for X4 shows where the OT mesh is anisotropic. \label{fig:distDxX2X4}}
\end{figure}

Scatter plots of the cell-centre to cell-centre distance, $|\mathbf{d}_{x}|$,
as a function of distance to the centre of the refined region are
shown in fig. \ref{fig:distDxX2X4} for the X2-X16 meshes of 2,562
cells on the sphere. This shows that the centroidal Voronoi meshes
are close to isotropic whereas the OT meshes have high anisotropy
where the second derivative of the monitor function is high. In particular,
a region of anisotropy is indicated by a blue ring for the X4 OT mesh
in fig. \ref{fig:distDxX2X4}: there is a wide range of $|\mathbf{d}_{x}|$
at the same distance to the centre of the refined region, indicating
anisotropy. This anisotropy could be reduced by smoothing the monitor
function.

Unlike the meshes on the plane, the meshes on the sphere are isotropic
in the uniformly coarse region, due to the isotropy of the domain
relative to the centre of refinement. This could be an advantage of
using r-adaptivity on the sphere over its use in Euclidean geometry
with corners. However the meshes on the sphere still have a bulge
in $|\mathbf{d}_{f}|$ on the edge of the coarse region. This is not
ideal for atmospheric simulations since global errors are often proportional
to the largest $|\mathbf{d}_{x}|$ \citep{RJG+11}.

\begin{figure}
\includegraphics[width=1\linewidth]{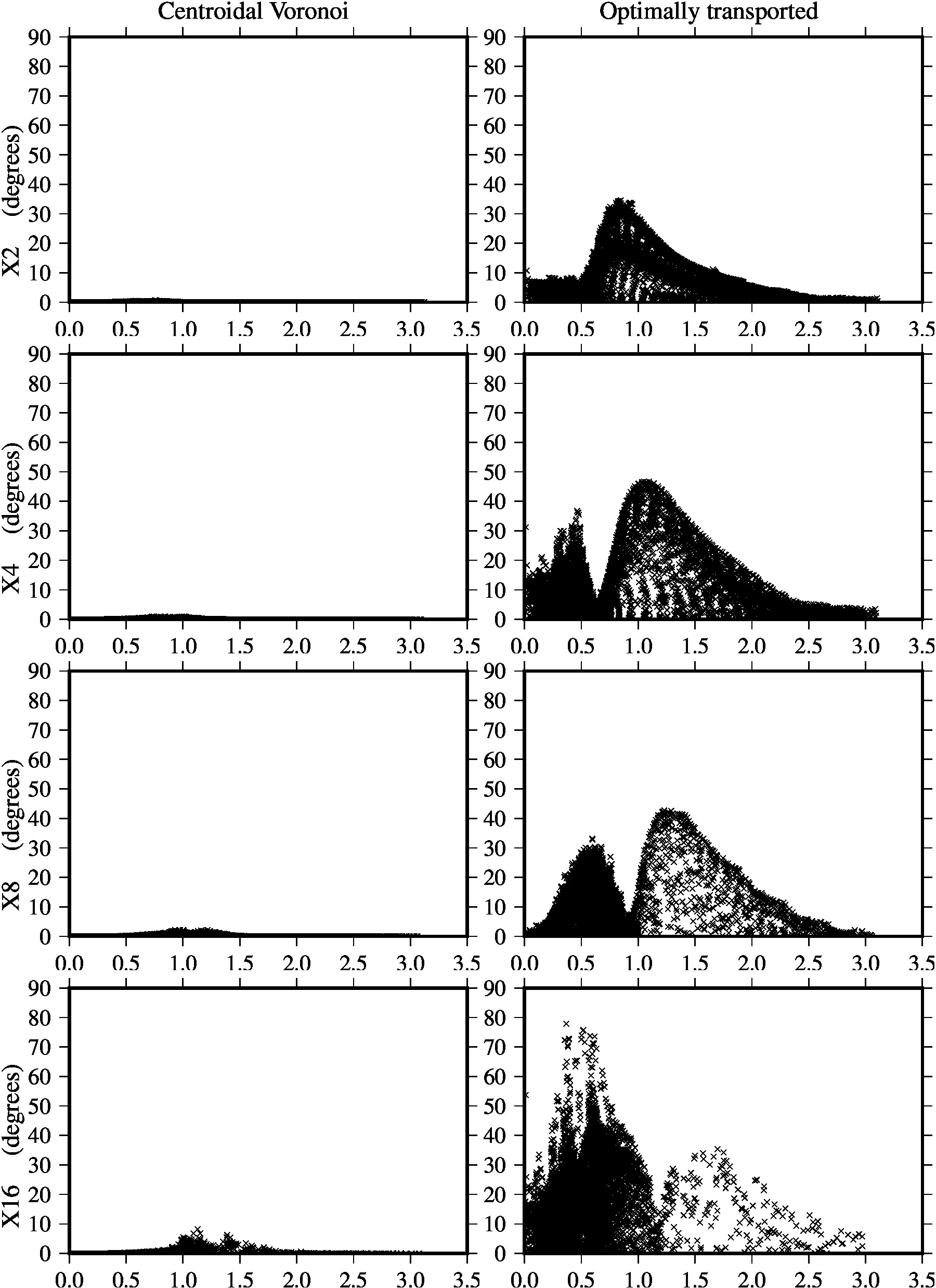}

\caption{Non-orthogonality of the meshes on the sphere (as defined in eqn.
\ref{eq:nonOrthog}) as a function of geodesic distance to the centre
of the refined region for the X2-X16 meshes of 2,562 cells\label{fig:meshOrthogX2X4}}
\end{figure}

\begin{figure}
\includegraphics[width=1\linewidth]{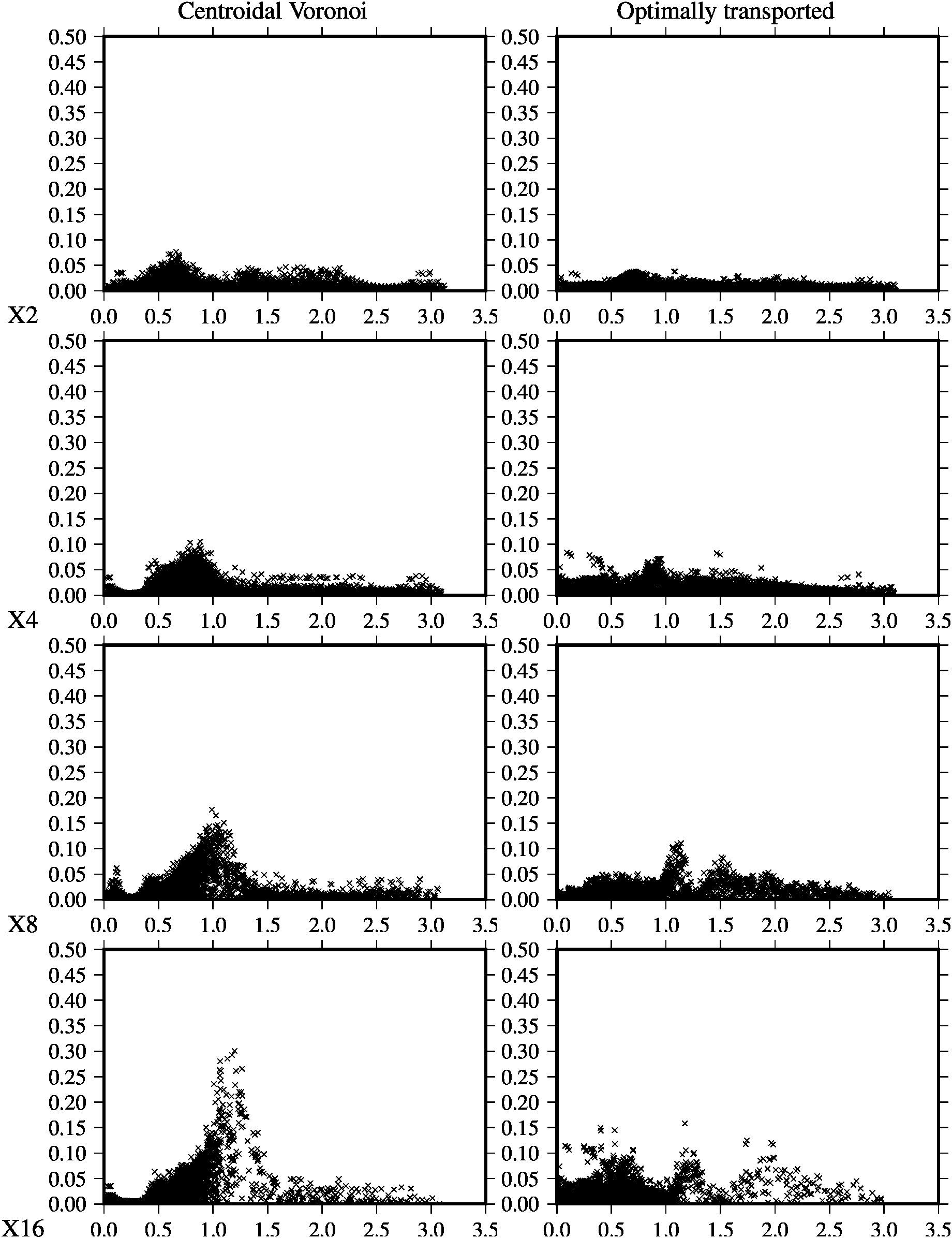}

\caption{Face skewness ($d_{s}/|\mathbf{d}_{x}|$, as defined in eqn. \ref{eq:skewness})
as a function of geodesic distance to the centre of the refined region
for the X2 and X4 meshes of 2,562 cells\label{fig:meshSkewX2X4}}
\end{figure}

The orthogonality and skewness of the faces of the OT X2-X16 meshes
on the sphere are shown in figs. \ref{fig:meshOrthogX2X4} and \ref{fig:meshSkewX2X4}
in comparison to the centroidal Voronoi meshes. Lloyd's algorithm
with a non-uniform monitor function generates exactly orthogonal,
non-centroidal meshes and so for comparison with the OT meshes, the
Voronoi meshes are made exactly centroidal at the expense of orthogonality
by using the cell centroid as the cell centre rather than using the
Voronoi generating point. Even so, they remain very close to orthogonal
in comparison to the OT meshes which have high non-orthogonality where
the second derivative of the monitor function is high (for this test
case). In fact the non-orthogonality reaches over 70 degrees for some
cells in the X16 mesh. This is unlikely to be a good mesh for simulation.
This problem will be investigated further in section \ref{sub:betterMeshQuality}.

The OT meshes have less face skewness, $d_{s}/|\mathbf{d}_{x}|$,
than the centroidal Voronoi meshes (fig \ref{fig:meshSkewX2X4}) which
could be advantageous for numerical methods whose errors depend on
skewness. For example, \citet{HR95b} described how to optimise orthogonal
meshes to reduce skewness for low-order finite-volume discritisations.

\subsection{Improving Convexity\label{sub:betterMeshQuality}}

The OT X16 meshes presented in sections \ref{sub:sphMeshes}-\ref{sub:meshQualitySph}
have some large non-orthogonality at regions where the resolution
is changing rapidly (fig \ref{fig:meshOrthogX2X4}). The reason for
this can be seen more clearly in a zoomed regions of the meshes in
the second row of fig. \ref{fig:X16zoom}. The double zoomed plot
shows that some of the cells are not convex. This implies that the
calculation of $\nabla_{v}\phi$ has in fact not yielded a smooth
vector field, despite the development of the Goldilocks stencil with
the aim of achieving a smooth $\nabla_{v}\phi$ on the smallest possible
stencil. The Goldilocks stencil does give a much smoother $\nabla_{v}\phi$
than the small stencil (first row of fig \ref{fig:X16zoom}). If instead
we interpolate $\nabla\phi$ from faces onto vertices which entails
the use of the larger stencil (secn \ref{par:wideGradv}), the non-convex
cells are not generated (third row of fig \ref{fig:X16zoom}). Alternatively,
a Voronoi tessellation can be created using the cell centres of the
Goldilocks stencil mesh as generating points (bottom row of fig. \ref{fig:X16zoom}).
This also eliminates non-convex cells. 

\begin{figure}
\includegraphics[width=1\linewidth]{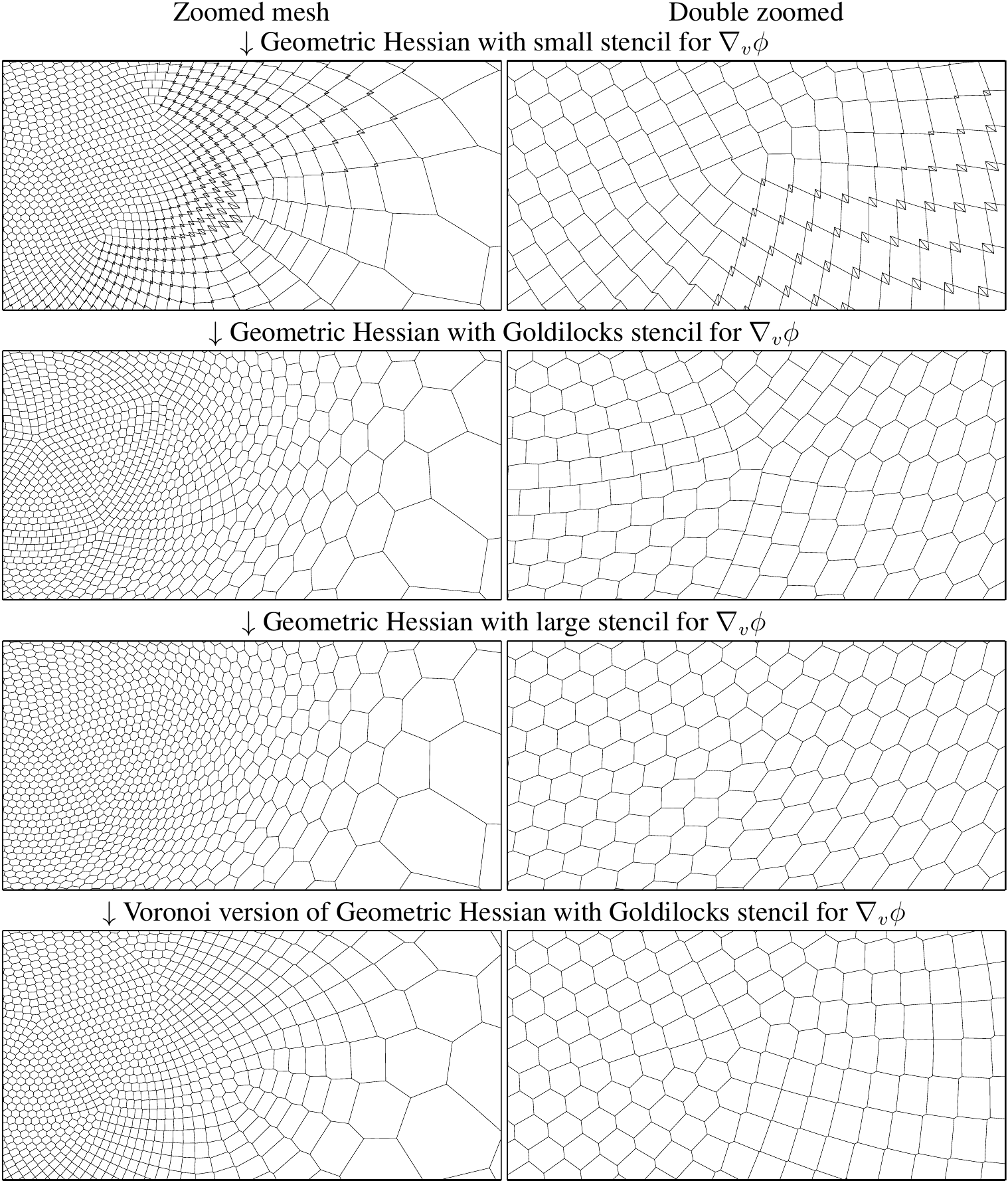}

\caption{Zoomed and double zoomed X16 meshes with 2,562 cells using different
calculations of $\nabla_{v}\phi$ and Voronoi versions.\label{fig:X16zoom}}
\end{figure}

\begin{figure}
\includegraphics[width=1\linewidth]{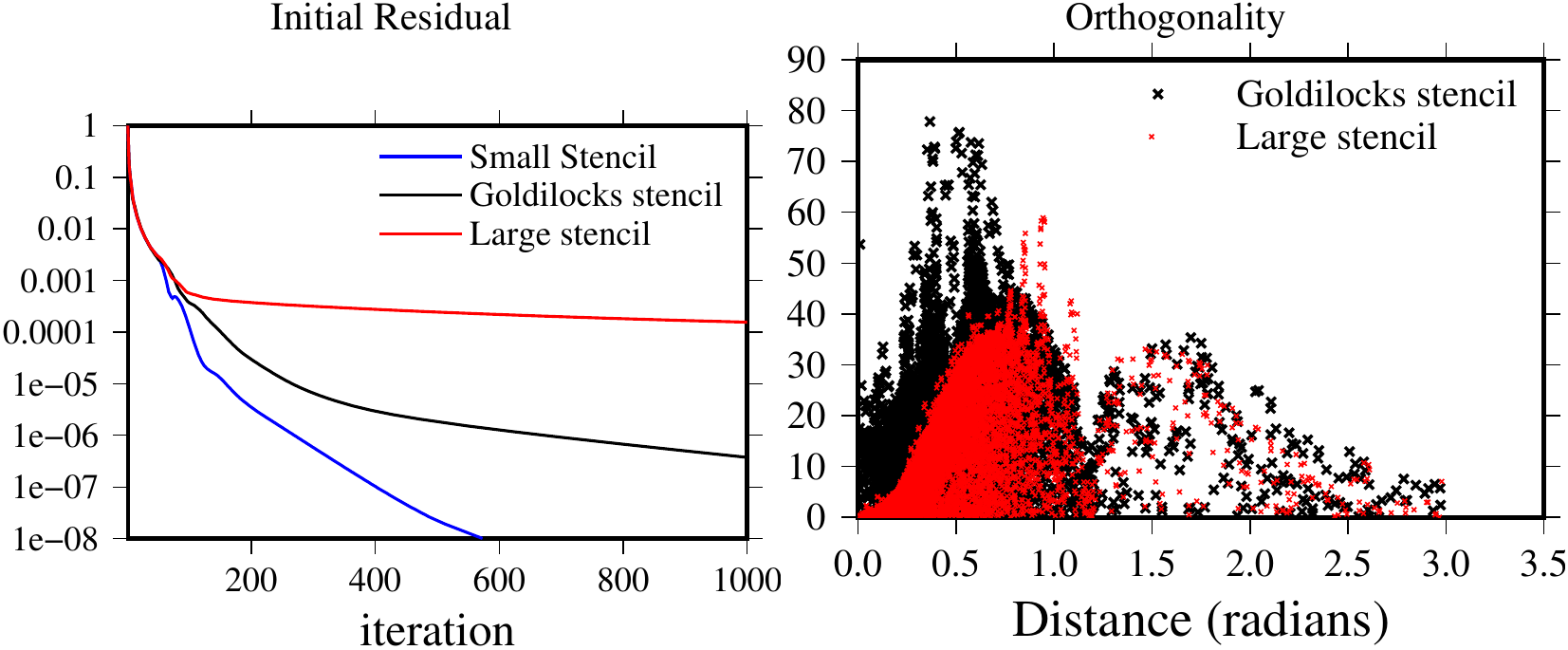}

\caption{Comparisons of convergence of initial residuals and orthogonality
for the X16 meshes of 2,562 cells calculated using the geometric Hessian
but with different calculations of $\nabla_{v}\phi$.\label{fig:gradp_5_16}}
\end{figure}

The problem with the large stencil calculation of $\nabla_{v}\phi$
is that convergence is slowed and orthogonality is only reduced a
little (fig \ref{fig:gradp_5_16}). Therefore it is necessary to consider
the Voronoi tessellation of the cell centres (final row of fig. \ref{fig:X16zoom}).
This modification does not affect the convergence since the Voronoi
tessellation is calculated after convergence of the Monge-Amp\`ere
solution. This mesh is insensitive to the calculation of $\nabla_{v}\phi$
but is no longer exactly equidistributed because the cell areas change
a little (locally) when the Voronoi tesselation is calculated (fig
\ref{fig:meshDiags_Voronoi_5_16}). However these area changes are
very small and simply smooth out the curve where the monitor function
flattens out into the coarse region. Fig \ref{fig:meshDiags_Voronoi_5_16}
also shows that the Voronoi version is more orthogonal than the large
stencil version, the anisotropy is similar and the skewness is increased.
However, the connectivity may be changed slightly.

\begin{figure}
\includegraphics[width=1\linewidth]{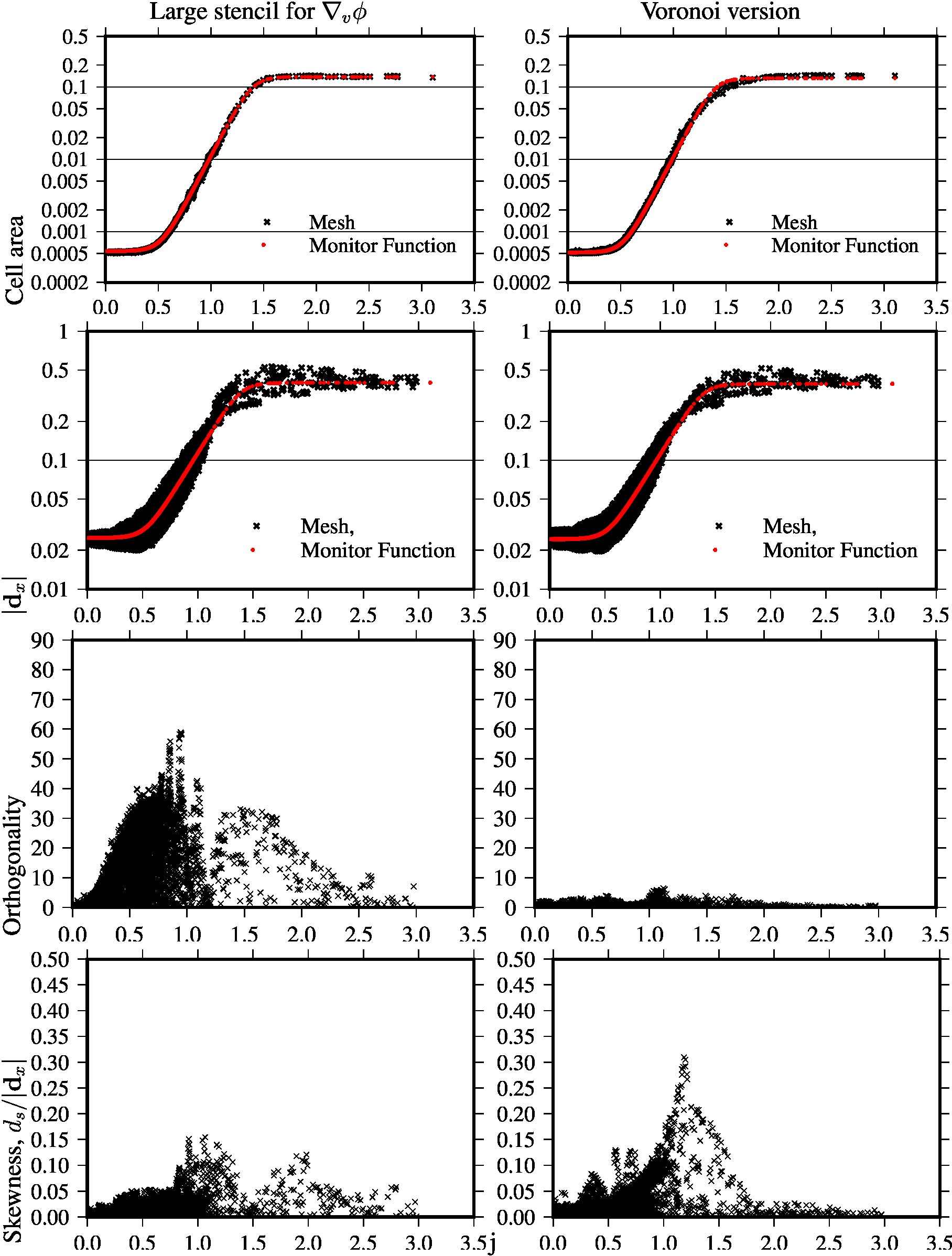}

\caption{Mesh diagnostics as a function of distance from the centre of the
refined region for the X16 meshes of 2,562 cells using the large stencil
for $\nabla_{v}\phi$ on the left and using the Voronoi tesselation
on the right. \label{fig:meshDiags_Voronoi_5_16}}
\end{figure}

\subsection{Optimally Transported Meshes using Precipitation as a Monitor Function\label{sub:ppt}}

In order to demonstrate the numerical solution of the Monge-Amp\`ere
type equation using realistic data as a monitor function, meshes are
generated based on the daily average precipitation rate from the NOAA-CIRES
20th Century Reanalysis version 2 (\citealp{CWS11}, \url{http://www.esrl.noaa.gov/psd/data/gridded/data.20thC_ReanV2.html})
on 9 Oct 2012. The numerical solution of the Monge-Amp\`ere equation
uses two near uniform hexagonal-icosahedral meshes of 2,562 and 10,242
cells. The re-analysis precipitation ranges from zero to $p_{\max}=8.73\times10^{-4}\text{kg}\text{m}^{-2}\text{s}^{-1}$.
A strictly positive, non-dimensional monitor function, $m$, is defined
from the precipitation rate, $p$ using:
\begin{equation}
m=\frac{p+p_{\min}}{p_{\max}+p_{min}}
\end{equation}
where \marginpar{Rev3.14}\sout{the minimum allowable values is
set to} $p_{\min}=10^{-5}\text{kg}\text{m}^{-2}\text{s}^{-1}$. The
resulting meshes are shown in fig. \ref{fig:pptMesh} (and are highly
sensitive to the value of $p_{\min}$ used). Precipitation clearly
could not be used as a monitor function for a dynamically adapting
simulation of the global atmosphere since it is strongly resolution
dependent. Instead, monitor functions with less resolution dependency
should be developed. Reanalysis precipitation is used here just as
a demonstration of the solution when using realistic meteorological
data.

\begin{figure}
\includegraphics[width=1\linewidth]{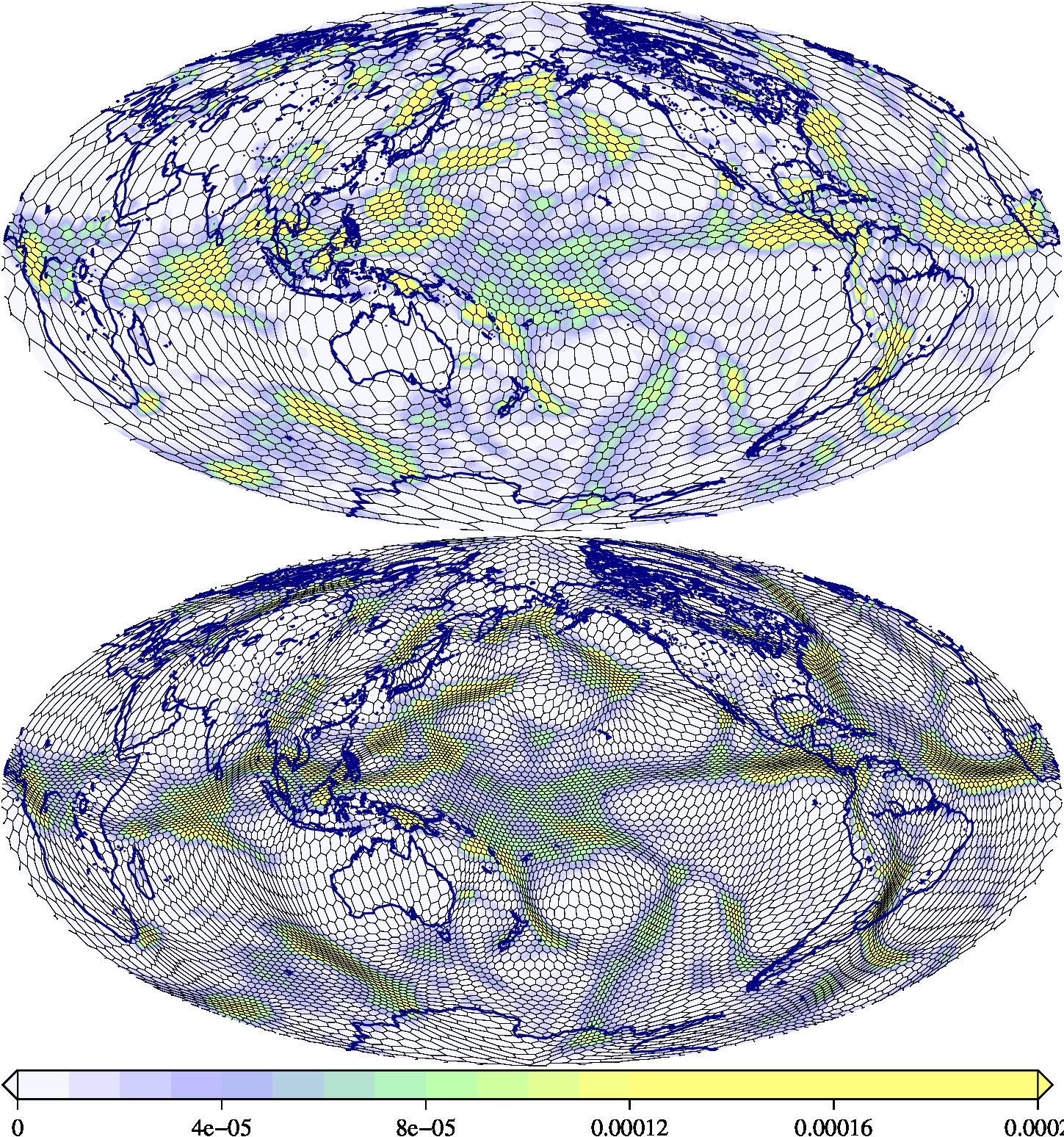}

\caption{Meshes of 2,562 and 10,242 cells generated based on the monitor function
of precipitation on 9 Oct 2012. The colours show the precipitation
rate in $\text{kg}\text{m}^{-2}\text{s}^{-1}$. The surface of the
entire sphere is displayed using a Hammer equal-area projection \citep{WSS+13}.
\label{fig:pptMesh}}
\end{figure}

The meshes resolved based on precipitation show excellent refinement
along fronts, particularly looking at South America. The Inter-tropical
convergence zone is also refined in the latitudinal direction. However,
based on the limitations of r-adaptivity, the Inter-tropical convergence
zone cannot be refined everywhere around the equator in the longitudinal
direction. If this were a requirement, a mesh starting with more points
around the equator should be used. This is the subject of future research.

\section{Conclusions\label{sec:concs}}

A technique for generating optimally transported (OT) meshes, solving
a Monge-Amp\`ere type equation on the surface of the sphere, has
been developed in order to generate meshes which are equidistributed
with respect to a monitor function. Equations of Monge-Amp\`ere type
have not before been solved numerically on the surface of a sphere.
We show that a unique solution to the optimal mesh transport problem
on the sphere exists and exponential maps are used to create the map
from the old to the new mesh. We introduce a geometric interpretation
of the Hessian rather than a numerical approximation which is accurate
on the surface of the sphere. In order to create a semi-implicit algorithm,
a new linearisation of the Monge-Amp\`ere equation is proposed which
includes a Laplacian term and the resulting Poisson equation is solved
at each fixed-point iteration.

To validate the novel aspects of the numerical method, we first reproduce
some known solutions of the Monge-Amp\`ere equation on a two dimensional
plane and find that the geometric interpretation of the Hessian leads
to more accurate equidistribution than a finite difference discretisation.
We also generate OT meshes of polygons on the sphere to compare with
the centroidal Voronoi meshes generated by \citet{RJG+11}. The geometric
Hessian \sout{created} accurately equidistributed meshes on the
surface of the sphere. The algorithm is found to be sensitive to the
numerical method used to calculate the gradient of the mesh potential
(the map to the new mesh) with a compact stencil leading to non-convexity
and a large stencil leading to very slow convergence. The mesh tangling
can be eliminated by creating a Voronoi tessellation of the cell centres
of the final mesh. The exact solution of the OT problem on the sphere
is c-convex which means that the mesh should not tangle. A numerical
method which reproduces this property will be the subject of future
work. 

The meshes generated have advantages and disadvantages relative to
centroidal Voronoi meshes generated using Lloyd's algorithm. In principle,
OT meshes should be much faster to generate, although we do not yet
have timing comparisons. OT meshes do not change their connectivity
with respect to the base, uniform mesh, so these meshes can be used
in r-adaptive simulations. In comparison to centroidal Voronoi meshes,
the OT meshes are non-orthogonal and less isotropic but have less
face skewness. In order to overcome the non-orthogonality of OT meshes,
the OT technique can be used to generate Voronoi meshes.

Finally, we generate a mesh using a monitor function based on reanalysis
precipitation. This mesh refines smoothly along atmospheric fronts
and convergence zones and provides inspiration for using r-adaptivity
for global atmospheric modelling. Suitable monitor functions for r-adaptive
simulations is also the subject of future work.

\section*{Acknowledgements}

Weller acknowledges support from NERC grant NE/H015698/1 and Browne
from NERC grants NE/J005878/1 and NE/M013693/1. Budd acknowledges
support the Pacific Institute for the Mathematical Sciences (PIMS)
who have funded his sabbaticals.

\bibliographystyle{abbrvnat}
\bibliography{numerics,bibfile}

\end{document}